\documentclass[11pt]{article}

 \usepackage[toc,page]{appendix}
\usepackage{appendix}
\usepackage[intoc]{nomencl}

\usepackage{minitoc}
\usepackage{changepage}

\usepackage[utf8]{inputenc}
\usepackage{listings}
\usepackage{amsmath, etoolbox}
\usepackage{tensor}
\usepackage{mathrsfs}

\makeatletter
\let\original@footnotemark\footnotemark
\newcommand{\align@footnotemark}{%
  \ifmeasuring@
    \chardef\@tempfn=\value{footnote}%
    \original@footnotemark
    \setcounter{footnote}{\@tempfn}%
  \else
    \iffirstchoice@
      \original@footnotemark
    \fi
  \fi}
\pretocmd{\start@align}{\let\footnotemark\align@footnotemark}{}{}
\makeatother
\usepackage{mathtools}
\numberwithin{equation}{subsection}

\usepackage{amsthm}
\usepackage{amssymb}
\usepackage{amsfonts}
\usepackage[font=it, labelfont=bf]{caption}
\usepackage{extarrows}
\usepackage{mathrsfs}
\usepackage{bm}
\usepackage[top=1.15in, bottom=1.15in, left=1.3in, right=1.3in]{geometry}
\usepackage{hyperref}
\usepackage{esint}
\usepackage[all,cmtip]{xy}
\usepackage{xcolor}
\usepackage{mathtools}
\usepackage{MnSymbol}
\usepackage{enumitem}
\usepackage[T1]{fontenc}
\usepackage{accents}

\newtheorem{thm}{{Theorem}}[section]

\newtheorem{prop}{{Proposition}}[section]

\newtheorem{lem}{{Lemma}}[section]
\newtheorem{rem}{{Remark}}[section]

\newtheorem{cor}{{Corollary}}[section]

\newtheorem*{que*}{Question}

\newtheorem*{Principle*}{Principle}

\newtheorem*{MetaTheorem*}{Meta Theorem}

\theoremstyle{plain}
\newtheorem{Main}{Theorem}

\newcommand{\Ad}{\mathrm{Ad}}

\newcommand{\SO}{\mathrm{SO}}
\newcommand{\SU}{\mathrm{SU}}
\newcommand{\GL}{\mathrm{GL}}
\newcommand{\SL}{\mathrm{SL}}
\newcommand{\M}{\mathrm{M}}

\newcommand{\R}{\mathbb{R}}
\newcommand{\Q}{\mathbb{Q}}
\newcommand{\C}{\mathbb{C}}
\newcommand{\Z}{\mathbb{Z}}
\newcommand{\N}{\mathbb{N}}

\newcommand{\fso}{\mathfrak{so}}

\newcommand{\fz}{\mathfrak{z}}

\newcommand{\diag}{\mathrm{diag}}

\newcommand{\Conj}{\mathrm{Conj}}
\newcommand{\End}{\mathrm{End}}

\newcommand{\T}{\mathbb{T}}

\newcommand{\cst}{\mathrm{cst}}

\newcommand{\RDC}{\mathrm{RDC}}

\def\cR{\mathcal{R}}
\def\cB{\mathcal{B}}

\def\cP{\mathcal P}

\def\cT{\mathcal T}

\def\DC{\mathrm{DC}}

\def\a{\alpha}
\def\e{\epsilon}
\def\ph{\varphi}
\def\ti{\tilde}
\def\la{\lambda}

\def\be{\begin{equation}}
\def\ee{\end{equation}}

\def\bm{\begin{pmatrix}}
\def\em{\end{pmatrix}}

\def\bsm{\begin{smallmatrix}}
\def\esm{\end{smallmatrix}}

\author{
Xuanji Hou\footnote{Hubei Key Laboratory of Mathematical Sciences, School of Mathematics and Statics, Huazhong Normal University, Wuhan 430079, China}, Yi Pan\footnote{Institute of Science and Technology Austria, Klosterneuburg 3400, Austria}, Qi Zhou\footnote{Chern Institute of Mathematics, NanKai University, Tianjin 300071, China}}

\title{Dynamical classification of analytic one-frequency quasi-periodic $\SO(3,\R)$-cocycles}

\date{}

\begin{document}
\maketitle

\abstract
We establish a close connection between acceleration and dynamical degree for one-frequency quasi-periodic compact cocycles, by showing that two vectors derived separately from each coincide. Based on this, we provide a dynamical classification of one-frequency quasi-periodic $\SO(3,\R)$-cocycles.

\section{Introduction}

The classification of dynamics is an important topic in dynamical systems. One well-known example is the classification of circle diffeomorphisms. Let $\T:=\R/\Z$ and $f: \T \rightarrow \T$ be an orientation-preserving $C^r$ homeomorphism, where $r\in\N\cup\{\infty,\omega\}$. Poincar\'e laid the foundations of the theory of circle diffeomorphisms and introduced the concept of rotation number in the classification of these diffeomorphisms. The rotation number $\rho(f)\in\R$ of $f$ is defined as the uniform limit $$\lim_{n\to\infty}\frac{{ f}^n(x)-x}{n}.$$ The rotation number is independent of the choice of $x\in\T$ and is invariant under conjugations, making it a useful dynamical invariant.
Poincar\'{e}'s classification theorem states that if the rotation number $\rho(f)$ is rational $\frac{p}{q}$, then $f$ has a periodic orbit of period $q$. If the rotation number is irrational, then $f$ is {\it semi-conjugate} to the irrational (rigid) rotation. Denjoy \cite{De} further showed that if the rotation number is irrational and $\ln Df$ has bounded variation, then the semi-conjugacy is a homeomorphism. This result was later extended by Arnold~\cite{Arnold}, as part of the KAM theory, who showed that if $f$ is $C^{\omega}$ close to a rotation and the rotation number $\rho(f)$ satisfies the Diophantine condition, then the conjugacy is analytic. Subsequent work by Herman, Yoccoz, Katznelson, Ornstein, Sinai, and others \cite{He79,He85,Y84, Y, KO89a,KO89b,KS87,SK89} extended this result to the global version of analytic, smooth, or finitely differentiable circle diffeomorphisms $f$. They showed that when $f$ is analytic (resp. $C^\infty$, $C^k$ with $k>2$) and $\rho(f)$ is Diophantine, the conjugacy is analytic (resp. $C^\infty$, $C^{k-r}$ for some $0<r<k$). Moreover, the Diophantine condition can be weakened to the Brjuno condition, which turns to be necessary \cite{Y}. Therefore, the dynamics of the circle diffeomorphism $f$ can be completely determined by its rotation number $\rho(f)$.

A one-frequency cocycle is a natural generalization of a circle diffeomorphism in high dimensions. Let ${\rm G}$ be a Lie-subgroup of ${\rm GL}(s, \C)$.
A  {\it  $C^r$ one-frequency  ${\rm G}$-cocycle}  $(\alpha,A)$ is a skew-product defined on $\T\times\mathbb C^s$ such that
$$(\alpha,A): \,\T\times \mathbb C^s\to \T\times \mathbb C^s,\quad (x,v)\mapsto (x+\alpha,A(x)\cdot v),$$
where $\alpha \in \R\backslash
\Q$, $A\in C^r(\T,{\rm GL}(s,\C))$, $r\in\N\cup\{\infty\}$. The iterates of the cocycle $(\alpha, A)^n=(n \alpha, A_n)$ are defined as   $A_{0}(x)=I$,
\begin{equation*}A_{n}(\cdot)=A(\cdot+(n-1)\alpha)\cdots A(\cdot), \quad n\ge 1,
\end{equation*} and $A_n(\cdot)=A_{-n}(\cdot-n\alpha)^{-1}$ for $n\leq -1$. 

\smallskip

When $G=\SO(2,\R)$, taking projective action, it naturally induces the cocycle
$$T_{\phi}: \T\times \T\to\T\times \T,\quad (x,g) \mapsto (x+\alpha, \phi(x)\cdot g) .$$
Moreover, if $\phi$ is smooth, then dynamical properties of $T_{\phi}$ depend on the topological degree $d(\phi)$ of $\phi$. For example, if  $\phi$ is $C^2$ with $d(\phi)\neq 0$, then   $T_{\phi}$ is ergodic and it has a countable Lebesgue spectrum on the orthocomplement of the space of functions depending only on the first variable \cite{ILR}. If $\phi$   is absolutely continuous with $d(\phi)= 0$, then  $T_{\phi}$ has a singular spectrum \cite{GLL}.

However, when  $G$ is non-compact, hyperbolicity appears.  Therefore, another important dynamical invariant is introduced.  For a cocycle $(\alpha,A)\in \R\backslash
\Q \times C^0(\T,{\rm GL}(s,\C))$,  we define the {\it Lyapunov exponents} $L_1(\alpha, A)\geq L_2(\alpha, A)\geq ...\geq L_s(\alpha,A)$  as
$$
L_k(\alpha,A):=\lim\limits_{n\rightarrow\infty}\frac{1}{n}\int_{\T}\ln(\sigma_k(A_n(x)))dx,\quad k=1,\cdots,s,
$$
(repeatedly according to multiplicities), where for a matrix $B$ we denote by $\sigma_1(B)\geq...\geq \sigma_s(B)$ its singular values  (eigenvalues of $\sqrt{B^*B}$).   As the k-th exterior product $\Lambda^kB$ of $B$ satisfies $\|\Lambda^kB\|=\sum_{j=1}^k\sigma_j(B)$, one      can define 
\begin{equation}\label{LEs}
L^k(\alpha,A):=\sum\limits_{j=1}^kL_j(\alpha,A)=\lim\limits_{n\rightarrow \infty}\frac{1}{n}\int_{\T}\ln(\|\Lambda^kA_n(x)\|)dx, \quad 
k=1,\cdots,s.
\end{equation}
Unfortunately, these Lyapunov exponents are far from enough to classify the cocycles. In fact, even for cocycles with all Lyapunov exponents vanishing, the dynamics still have a variety of possibilities \cite{E92,Kri02}.

\subsection{Accelerations}
A breakthrough came recently, with the establishment of the global theory of one-frequency analytic Schrödinger operators. Avila \cite{A15}
classified one-frequency analytic $\SL(2,\R)$ cocycles and proposed another dynamical invariant called "accelerations". Indeed, for any analytic cocycle $(\alpha,A)\in \R\backslash \Q \times C^{\omega}(\T,{\rm GL}(s,\C))$, there exists $\delta>0$ such that $A$ can be holomorphically extended to $\{|\Im z|\leq\delta\}$. Let $A_\e(x):= A(x+i\e)$ for $|\e|<\delta$. We can then define the {\it accelerations} of $(\alpha,A)$ as follows:
\begin{equation}
	\omega^k:=\lim\limits_{\e\rightarrow 0^+}\frac{1}{2\pi\e}(L^k(\alpha,A_\e)-L^k(\alpha,A)), \qquad
	\omega_k:=\omega^k-\omega^{k-1}.
\end{equation}
Furthermore, we define $(\omega_1,\cdots,\omega_s)$ as the {\it acceleration vector} of $(\alpha,A)$. The most important observation is that the accelerations are quantized \cite{A15,AJS14}, {namely} that there exists $l\in \{1,\cdots, s\}$ such that both $ l \omega_k$ and $l\omega^k$ are integers. 
Actually the accelerations are dynamical invariants as they remain invariant under conjugations. We recall that $(\alpha, A_i)\in \R\backslash \Q \times C^{\omega}(\T,G)$, $i=1,2$, are conjugate to each other, if there exists $B\in C^\omega( \chi_{G}\T, G)$\footnote{$\chi_G\in\{1,2\}$. In particular, for $G=\SO(3,\R)$, $\chi_G=1$.}   such that $$A_1(\cdot)=B(\cdot +\alpha) A_2(\cdot) B(\cdot)^{-1}. $$

With these notations, Avila's classification of one-frequency analytic $\SL(2,\R)$-cocycles is introduced \cite{A15}. If $A$ is homotopic to the identity, or in other words, if the topological degree $d(A)=0$, then the dynamical behavior can be determined by the Lyapunov exponent and acceleration. To be precise, if $\omega_1 = 0$, then the cocycle is subcritical (resp. uniformly hyperbolic) when $L_1(\alpha,A) = 0$ (resp. $L_1(\alpha,A) > 0$). On the contrary, if $\omega_1 \neq 0$, $(\alpha,A)$ is critical (resp. supercritical or non-uniformly hyperbolic) when $L_1(\alpha,A) = 0$ (resp. $L_1(\alpha,A) > 0$). On the other hand, if the topological degree $d(A) \neq 0$, i.e., $A\colon$ $\T \to{\rm SL}(2,\R)$ is homotopic to $\theta \mapsto R_{d(A)\theta}$, where

\[ R_\theta:=
\begin{pmatrix}
	\cos2 \pi\theta & -\sin2\pi\theta\\
	\sin2\pi\theta & \cos2\pi\theta
\end{pmatrix},
\]
Avila-Krikorian \cite{AK15} proved that if $(\alpha,A)$ is $L^2$-conjugate to rotations, then it is $C^{\omega}$-conjugate to rotations. In conclusion, the classification of analytic $\SL(2,\R)$-cocycles depends on three dynamical invariants: the topological degree, Lyapunov exponent, and acceleration.

The natural inquiry is whether this kind of classification can be generalized to high-dimensional cocycles. As an initial attempt, we consider this question for cocycles valued in ${\rm SO}(s,\R)$, $s\in\N^*$. Note that if $A$ takes values in $\SO(s,\R)$, then $L_k(\alpha,A)=0$, $k=1,\cdots,s$. Hence the monotonicity on singular values therefore Lyapunov exponents provides that
$$\omega_1\geq\cdots\geq\omega_s.$$  In particular, for any $\SO(3,\R)$-cocycle,  as determinant is identically $1$, we have  
$\omega_1+\omega_2+\omega_3=0.$
As singular values $\la,\la^{-1}$ come in pairs, we deduce that nonzero accelerations are in pairs $\omega,-\omega$. Hence the acceleration vector must be in the form
\be
(\omega_1,0,-\omega_1),
\ee
i.e., essentially has one variable. We find that, contrary to the $\SL(2,\R)$ case, the acceleration determines the dynamics of analytic one-frequency ${\rm SO}(3,\R)$ cocycles.

\begin{Main}
	\label{thm.ac.2}
	Let $\a\in \R\backslash\Q$ and $A\in C^\omega(\T, {\rm SO}(3,\R))$, then we have the following:
	\begin{enumerate}
		\item $\omega_1=0$ if and only if $(\a,A)$ is $C^\omega$-almost reducible, i.e., the closure of the analytic conjugate class of $(\a,A)$ contains a constant cocycle.
		\item Assuming additionally that $\a\in \RDC$, if $\omega_1 \neq 0$, then $(\a,A)$ is $C^\omega$-conjugate to $$(\a,\exp\bm0& \omega_1  x+c&0\\-\omega_1 x-c&0&0\\0&0&0\em),$$ for some $c\in\R$.
	\end{enumerate}
\end{Main}

We say $\alpha\in\R\backslash\Q$ satisfies a \textit{Diophantine condition} DC($\kappa,\tau$) if $\kappa>0$ and $\tau>0$, where

\[ ||q\alpha-p||_\Z>\frac{\kappa}{|q|^{\tau}}, ~(p,q)\in\Z^2, ~q\neq0,
\]
where $||\cdot||_\Z$ represents the distance to the nearest integer for a real number. Furthermore, we say that $\alpha \in \R\backslash\Q$ is {\it recurrent Diophantine} if there exist some $\gamma>0,\tau>1$ such that $G^n(\alpha)\in \DC(\gamma,\tau)$ for infinitely many $n$, where $G: (0,1)\rightarrow (0,1)$ is the Gauss map defined by $G(x)= \{\frac{1}{x}\}$ and $\{x\}$ denotes the fractional part of $x$. Let $\RDC$ be the set of recurrent Diophantine numbers, which has a full Lebesgue measure. We point out that the condition $\alpha \in \RDC$ is used for technical reason. It would be interesting to generalize the result to all $\alpha\in \R\backslash\Q$.

Theorem \ref{thm.ac.2} provides a comprehensive understanding of analytic one-frequency $\SO(3,\R)$-cocycles. This theorem is particularly interesting for several reasons. Firstly, Theorem \ref{thm.ac.2} (1), generalizes Avila's well-known Almost Reducibility Conjecture (ARC) to high-dimensional compact groups. The ARC states that if $(\alpha,A) \in \R\backslash\Q \times C^{\omega}$ $(\T,{\rm SL}(2,\R))$ is subcritical, then it is analytically almost reducible. This conjecture has important implications for both dynamical and spectral aspects, as highlighted in various works \cite{A3,A15,AJM,AKL,AYZ, GY,MJ}. Notably, our proof of the generalized ARC in Theorem \ref{thm.ac.2} is completely different from the proof in the case of $\SL(2,\R)$ \cite{A3,A2}. Our method is applicable to  compact group, demonstrating that ARC holds for general compact groups. For the sake of conciseness, we present the result in the case of $\SO(3,\R)$.

Moreover, Theorem \ref{thm.ac.2} provides classification  for one-frequency analytic $\SO(3,\R)$-cocycle: all cocycles can be classified into almost reducible ones  and non-almost-reducible ones,
 since the normal form cocycle in Theorem \ref{thm.ac.2} (2) is obviously not almost-reducible. Furthermore, Theorem \ref{thm.ac.2} (1) states that almost reducibility  holds for all irrational frequency (not merely $\RDC$ as in \cite{Fra04,K01}), which is a completely new result.

%

Finally, we would like to mention that for the case $\omega_1=0$, the result should hold for general compact groups. However, when $\omega_1\neq 0$, the rigidity result, as demonstrated in Theorem \ref{thm.ac.2} (2), can only be expected in a weaken version (achieving by a sequence of conjugations) for general compact groups. This is due to extra difficulty solving cohomological equation, arising in the possibility of the matrix form of dynamical degree being degenerate and having a large centralizer. One can consult \cite{P23} for more discussions. 

\subsection{Dynamical Degree}

As mentioned earlier, for $G=\SO(2,\R)$ or $\SL(2,\R)$, where the fundamental group of $G$ is nontrivial, the topological degree plays a significant role in dynamical classification. However, for $G=\SO(s,\R)$, $s\geq3$, the fundamental group is $\Z/2\Z$, rendering the topological degree less applicable. In contrast, another degree called the "dynamical degree" can be defined and utilized. The concept of dynamical degree was initially proposed by Krikorian \cite{K01} and formally introduced by Fr\k{a}czek~\cite{Fra04} for $\SU(2)$-cocycles. It can be extended to any compact cocycles and general cocycles $L^2$-conjugate to rotations \cite{AKP,Ka15,P23}.

Let $\alpha \in \mathbb{R}\backslash\mathbb{Q}$ and $A\in C^1(\mathbb{T},  \text{SO}(s,\mathbb{R}))$, $s\geq3$, $\mathcal{A}_\alpha$ be the operator defined as following:
$$\mathcal{A}_\alpha: L^2(\mathbb{T}, \mathfrak{so}(s,\mathbb{R}))\longleftrightarrow L^2(\mathbb{T}, \mathfrak{so}(s,\mathbb{R})), \quad Y\mapsto  A^{-1}(\cdot) Y(\cdot+\alpha) A(\cdot),$$
and $\mathcal{A}_\alpha^j$ represent its $j$-th iteration. It is known that $\mathcal{A}_\alpha$ is a unitary operator. According to von Neumann's ergodic theorem, there exists an $\mathcal{A}_\alpha$-invariant $D\in  L^2(\mathbb{T}, \mathfrak{so}(s,\mathbb{R}))$, such that the following convergence holds:
\begin{eqnarray}\label{L2-limit}
	\frac{1}{n} A_n(\cdot)^{-1}\partial A_n(\cdot)=\frac{1}{n}\sum_{j=0}^{n-1}\mathcal{A}_\alpha^j(A(\cdot)^{-1}\partial A(\cdot))
	\xrightarrow{L^2} D,
\end{eqnarray}
and $A(x)^{-1}D(x+\alpha)A(x)=D(x)$ almost everywhere. We call any matrix $D$ in the algebraic conjugate class the {\it dynamical degree} of $(\alpha,A)$. The dynamical degree is invariant under $C^1$-conjugation. Since $D\in\fso(s,\mathbb{R})$, there exist $Q\in\text{SU}(s)$ and $d_1,\cdots,d_s\in\mathbb{R}$ such that $D$ is conjugate to the diagonal form
$$
D=Q\cdot2\pi i\text{diag}\left(d_1,\cdots,d_s\right)\cdot Q^{-1}.
$$
Furthermore, notice that conjugation of $\bm0&1\\-1&0\em\in\SO(2,\R)$ on diagonal matrix acts like a transposition on diagonal elements, we can properly order the diagonal entries of $D$ such that
$$
d_1\geq \cdots\geq d_s.
$$
We thus define the {\it vector of dynamical degree} to be 
$\vec{d}=(d_1,\cdots,d_s).$

\smallskip

We note that the dynamical degree is quantized ($d_k\in\mathbb{Z}$, $k=1,\cdots,s$), as proven in \cite{AKP, P23}.
Since acceleration is also quantized \cite{A15, AJS14}, a natural question arises as to whether there is any relation between them. In fact, for analytic $\text{SO}(s,\mathbb{R})$-cocycles, they are the same.

\begin{Main}
	\label{thm.ac.1}
	Let  $\alpha \in \mathbb{R}\backslash\mathbb{Q}$ and $A\in C^\omega (\mathbb{T}, \SO(s,\mathbb{R}))$. The acceleration vector of $(\alpha,A)$ is the same as the vector of dynamical degree of $(\alpha,A)$.
\end{Main}

As the vector of dynamical degree has integer entries, we deduce the following result.
\begin{cor}
	Let $\alpha \in \mathbb{R}\backslash\mathbb{Q}$ and $A\in C^\omega (\mathbb{T}, \SO(s,\mathbb{R}))$. The accelerations of $(\alpha,A)$ are all integers.
\end{cor}
We recall \cite[Theorem 1.4]{AJS14}: when $A\in C^\omega(\mathbb{T},\mathcal{L}(\mathbb{C}^s,\mathbb{C}^s))$, for $k=1,\cdots,s$, there exists an integer $1\leq l\leq s$, such that $lw_k\in\mathbb{Z}$. Our improvement is essentially due to the compactness of $\text{SO}(s,\mathbb{R})$: eigenvalues always have algebraic multiplicity $1$, which leads to $l=1$.

\medskip

Consequently, the dynamical degree can be used to classify analytic one-frequency quasi-periodic $\text{SO}(3,\mathbb{R})$-cocycles. In particular, we can establish the relation between acceleration, dynamical degree, and almost reducibility.
\begin{cor}
	\label{thm.ac.3}
	Let $\alpha \in \mathbb{R}\backslash\mathbb{Q}$ and $A \in C^\omega (\mathbb{T}, \SO(3,\mathbb{R}))$. Then the following statements are equivalent:
	\begin{enumerate}
		\item The dynamical degree is zero,
		\item The largest acceleration $\omega_1$ is zero,
		\item $(\alpha,A)$ is $C^\omega$ almost reducible.
	\end{enumerate}
\end{cor}

\subsection{More Histories and Comments}

A \textit{quasi-periodic linear system on a Lie group $G$} (referred to as a linear system for short) is an ordinary differential equation on $\mathbb{T}^d \times G$ of the form
\begin{align}
	\label{loc.cont-1}
	\begin{cases}
		\dot{x} = \tilde{A}(\theta)x,\\
		\dot{\theta} = \omega,
	\end{cases}
\end{align}
where $\tilde{A}:\mathbb{T}^d \rightarrow \mathfrak{g}$ ($\mathfrak{g}$ is the Lie algebra of $G$), and $\omega \in \mathbb{T}^d$ is rationally independent. Poincar\'e time-$1$ map of \eqref{loc.cont-1} is in fact a quasi-periodic cocycle on $G$. Naturally, conjugation and in particular almost reducibility results of the corresponding cocycle is closely related to the ones of the original linear system \cite{YZ13}, cf. section~\ref{change.to.lin.s.sect}.

Let the frequency $\alpha \in \mathbb{T}^d$ be Diophantine and $A \in C^{\omega}(\mathbb{T}^d, \SL(2,\mathbb{R}))$. In 1975, Dinaburg-Sinai \cite{DS75} used the KAM approach to prove that for a typical one-parameter family of cocycles close to a constant, reducible cocycles are of positive measure. Subsequently, Eliasson \cite{E92} made a significant breakthrough by proving, under the same assumption, that reducible cocycles are of full measure, through an essential improvement of the KAM approach. By adapting and improving Eliasson's approach, Krikorian \cite{K99a} extended the full-measure reducibility result in \cite{E92} to the cases of semi-simple compact Lie groups,
and was further generalized to $\GL(s,\R)$ cocycles  quite recently by Wang-Xu-You-Zhou \cite{WXYZ}. 
All of these results hold for smooth or finitely differentiable cases, regardless of whether we consider linear systems or cocycles \cite{Am09, CCYZ19, FK09, K99b}.

However, the Diophantine condition on frequencies in local problems can be relaxed when considering one-frequency cocycles or two-frequency linear systems on $\SL(2,\mathbb{R})$, as shown in  \cite{A3, AFK11,HY12}. Both of these articles developed non-standard KAM techniques to deal with conjugation problems. Although the method developed in \cite{HY12} only applies to linear systems, the result still holds for cocycles due to the local embedding theorem \cite{YZ13}.

One may inquire whether any conclusions can be derived regarding conjugation for cocycles or linear systems that are not necessarily close to constants (referred to as the global case). In the case of one-frequency cocycles, global results can be anticipated.
In \cite{Kri02}, Krikorian proved a profound result stating that for a recurrent Diophantine frequency, $C^\infty$-reducible smooth cocycles valued in $\SU(2)$ are $C^\infty$-dense. Moreover, Krikorian proved that for a recurrent Diophantine frequency, non-reducible smooth cocycles on $\SU(2)$ can be divided into countable conjugation classes. In fact, the proof in \cite{Kri02} indicates the existence of a quantified conjugation-invariant quantity, which is precised later by Fraczek \cite{Fra04}.   In the case of $\SL(2,\mathbb{R})$, Avila-Krikorian \cite{AK06} proved that given recurrent Diophantine frequency, for almost every energy, the corresponding Schr\"odinger cocycle is either $C^{\omega}$-reducible or non-uniformly hyperbolic, based on Kotani theory \cite{Ko84,Si83} and the renormalization scheme introduced in \cite{Kri02}.

Furthermore, several interesting results have been established in \cite{Ka15,Ka17,Ka18,AKP}, which include topics such as degree, reducibility, almost-reducibility, non-reducibility, and spectrum of one-frequency cocycles on semi-simple compact groups. Avila's global theory and accelerations in \cite{A15,AJS14} for analytic one-frequency cocycles provide new insights into understanding global conjugation problems.

It is important to note that most of the conclusions about cocycles on semi-simple compact groups, whether local or global, require Diophantine or recurrent Diophantine conditions. In this paper, we aim to examine the relations among accelerations, dynamical degree, and almost-reducibility for analytic cocycles over any irrational frequency. We firstly recall the convergence of renormalization proved in\cite{AKP, P23}. Then by proving the preservation of acceleration along renormalization and using the explicit form of limit of renormalization, we deduce Theorem \ref{thm.ac.1} concerning coincidence of acceleration vector and vector of dynamical degree. As for Theorem \ref{thm.ac.2}, if $\omega_1=0$, we conclude by showing a local almost-reducibility result over any irrational frequency. We prove such result using the KAM scheme developed by Hou-You \cite{HY12} and the local embedding theorem \cite{YZ13}, which are applicable only to the analytic case. On the other hand, if $\omega_1\neq0$, we prove conjugation result through an adapted KAM scheme developed by Krikorian \cite{K01} for $\SU(2)$-cocycles. A technical point in the function of this KAM scheme is the estimates on length, which were developed in \cite{K01,Ka15} for compact cocycles. Since accelerations can only be defined for analytic cocycles, Corollary~\ref{thm.ac.3} holds only for the analytic case.

Finally, let us mention that the dynamics of cocycles has received significant attention from researchers, with a focus on the conjugation problems and their classifications. These questions are not only important in the field of dynamical systems, but also in mathematical physics, where they have wider applications in the spectral theory of Schr\"odinger operators \cite{A3,A15,AJM,AKL,AYZ, GY,MJ, You}.


\section*{Acknowledgement}
X. Hou is partially  supported by NNSF of China (Grant 12071083) and Funds for Distinguished Youths
of Hubei Province of China (2019CFA680). Y. Pan is supported by ERC Advanced Grant (\#885707).  Q. Zhou is partially supported by   National Key R\&D Program of China (2020YFA0713300), NSFC grant (12071232)  and Nankai Zhide Foundation.

\section{Preliminary}

\subsection{Lie group and Lie algebra}
Let
$$\SO(s,\R):=\{X\in\M(s,\R)~:~X^TX=XX^T=I\},$$
$\fso(s,\R)$ be the Lie algebra of $\SO(s,\R)$, namely
$$\fso(s,\R):=\{X\in\M(s,\R)~:~X=-X^T\}.$$
For a matrix $D\in\fso(s,\R)$, its {\it centralizer} $\fz(D)$ is defined to be
$$\fz(D)=\{C\in\fso(s,\R)~|~DC=CD\}.$$

In this article, we mainly work on $\SO(3,\R)$ and $\fso(3,\R)$. As for the Lie algebra $\fso(3,\R)$, it has a basis
$$J_1=\bm0&1&0\\-1&0&0\\0&0&0\em,
J_2=\bm0&0&1\\0&0&0\\-1&0&0\em,
J_3=\bm0&0&1\\0&0&0\\-1&0&0\em,$$
satisfying
$$[J_1,J_2]=J_3,\quad [J_2,J_3]=J_1,\quad [J_3,J_1]=J_2.$$
Then we have the following observation.
\begin{lem}
\label{cen}
For $D\in\fso(3,\R)\backslash\{0\}$, its centralizer is the $1$-dimensional linear subspace generated by $D$.
\end{lem}
Besides, it is easy to see that for $A\in\fso(3,\R)$, it has eigenvalue $\{||A||i,-||A||i,0\}$ and there exists $B\in\SO(3,\R)$ such that
$$B^TAB=||A||J_1.$$

\subsection{Continued fraction expansion}
\label{sect.alpha}
Define as usual for $0<\alpha<1$,
\begin{equation*}
a_0=0, ~~\alpha_0=\alpha,\quad q_0=1,\quad p_0=0,\quad q_{-1}=0,\quad p_{-1}=1,
\end{equation*}
and inductively for $n\geq1$,
\begin{equation*}
a_n=[\alpha_{n-1}^{-1}],~~\alpha_n=\alpha_{n-1}^{-1}-a_n=G(\alpha_{n-1})=\{\frac{1}{\alpha_{n-1}}\},
\end{equation*}
\be
\label{def.best.approx}
p_n=q_np_{n-1}+p_{n-2},\quad q_n=a_nq_{n-1}+q_{n-2},
\ee
where $[\cdot]$ denotes the integer part, $\{\cdot\}$ denotes the fractional part and $G$ is the Gauss map $G(x)=\{x^{-1}\}$. Let
$$
\beta_n=\Pi_{j=0}^n\alpha_j,
\quad
U(x)=
\left(\begin{matrix}0&1\\1&-[x^{-1}] 
\end{matrix}\right).$$
Define
\begin{equation*}
Q_0=I,\quad Q_n=U(\alpha_{n-1})\cdots U(\alpha_0).
\end{equation*}
Then \eqref{def.best.approx} shows that
\begin{equation*}
Q_n=(-1)^n\left(\begin{matrix}q_n&p_n\\q_{n-1}&p_{n-1}
\end{matrix}\right).
\end{equation*}
And we have
\begin{equation*}
\beta_n=(-1)^n(q_n\alpha-p_n)=\frac{1}{q_{n+1}+\alpha_{n+1}q_n},
\end{equation*}
\begin{equation*}
\frac{1}{q_{n+1}+q_n}<\beta_n<\frac{1}{q_{n+1}}.
\end{equation*}

\subsection{Renormalization of actions and cocycles}
\label{sect.Z2}
$\mathbb{Z}^2$-actions and renormalization of actions will be a fundamental tool in this paper. We recall here, following \cite{AK06}, $\Z^2$-action and the scheme of renormalization introduced in \cite{Kri02}. In particular, we focus on analytic case for $\SO(s,\R)$, $s\geq3$.

 Denote  by  $\Omega^\omega=\R\times C^\omega(\R,\SO(s,\R))$ the group composed of skew-product diffeomorphisms $(\alpha, A):\R\times \SO(s,\R)\to\R\times \SO(s,\R),$
\be
(\alpha, A)(x,v)=(x+\alpha, A(x)v),
\ee
with the composition being the group operation. An \textit{analytic fibered $\Z^2$-action} is a homomorphism $\Phi:\Z^2\to\Omega^\omega$. We use $\Lambda^\omega$ to denote the space of  analytic  fibered $\Z^2$-actions, endowed with the pointwise topology.

For two elements in $\Omega^r$, $(\a_1,A_1), (\a_2,A_2)$, they form a {\it commuting pair}, denoted by $\bm(\a_1,A_1)\\(\a_2,A_2)\em$, if they are commuting with each other, i.e.,
$$A_1(x+\a_2)A_2(x)=A_2(x+\a_1)A_1(x).$$
As $\Phi(1,0)$ and $\Phi(0,1)$ determine $\Phi$ and commute with each other by definition, they form a commuting pair. In below, we will not distinguish an $\Z^2$-action $\Phi$ and the corresponding commuting pair $(\bsm \Phi(1,0)\\\Phi(0,1)\esm)$.

For any $\Phi\in \Lambda^\omega$, we define $\gamma^\Phi_{n,m}:=\Pi_1\circ\Phi(n,m)\in\R$ and $A^\Phi_{n,m}:=\Pi_2\circ\Phi(n,m)\in C^\omega(\R,\SO(s,\R))$ for all $(n,m)\in\Z^2$, where
 $\Pi_1: \R\times C^\omega(\R,\SO(s,\R))\to\R$ and $\Pi_2: \R\times C^\omega(\R,\SO(s,\R))\to C^\omega(\R,\SO(s,\R))$ are coordinate projections. Let $\Lambda_0^\omega$  be the set of $\Phi\in\Lambda^\omega$ with $\gamma_{1,0}^\Phi=1$.

Notice that $C^\omega(\R,\SO(s,\R))$ acts on $\Omega^\omega$ by
$\Conj_B(\a,A(\cdot))=(\a,B(\cdot+\a)A(\cdot)B(\cdot)^{-1}).$
This action extends to an action, still denoted $\Conj_B$, on $\Lambda^\omega$. For $\Phi\in\Lambda^\omega$, the $\Conj_B(\Phi)$ is given by
$$\Conj_B(\Phi)(n,m)=(\gamma^\Phi_{n,m},B(\cdot+\gamma_{n,m}^\Phi)A_{n,m}^\Phi(\cdot)B(\cdot)^{-1}),\quad n,m\in\Z.$$
 We say that $\Phi,\Phi'\in\Lambda^\omega$ are \textit{$C^\omega$-conjugate} via $B\in C^{\omega}(\R,\SO(s,\R))$ if
$\Phi'=\Conj_B(\Phi)$ holds.

 We say that a $\Z^2$-action $\Phi$ is \textit{normalized} if $\Phi(1,0)=(1,I)$. In this case, $\Phi(0,1)=(\alpha, A)$ can be viewed as an analytic cocycle, since $A$ is automatically $1$-periodic. Conversely, given an analytic cocycle $(\alpha,A)$, we associate a normalized action $\Phi_{\alpha,A}$ by setting $
\Phi_{\alpha, A}(1,0)=(1,I),~~\Phi_{\alpha, A}(0,1)=(\alpha, A).
$

We cite the normalizing lemma from \cite{AK06}, whose proof remains true for $\Z^2$-actions valued in compact groups. We will provide a quantitative version of it in section~\ref{sect.5.2}.
\begin{lem}
\label{Normalization}
For  any $\Phi\in\Lambda_{0}^\omega$, $\Phi$ is $C^\omega$-conjugate to a normalized action.
\end{lem}

Let $\Phi\in\Lambda_0^\omega$, $B\in C^\omega(\R,\SO(s,\R))$ such that $\Conj_B(\Phi)$ is normalized. Then we call $B$ a \textit{normalizing map} of $\Phi$. The choice of $B$ may not be unique.

To define renormalization, we start with $\Z^2$-actions. We introduce the following maps.
\begin{enumerate}
\item Let $\lambda\neq0$. Define \textit{rescaling} $M_\lambda: \Lambda^\omega\to\Lambda^\omega$ by
\begin{equation}
M_\lambda(\Phi)(n,m):=(\lambda^{-1}\gamma_{n,m}^\Phi, A^\Phi_{n,m}(\lambda\cdot)),\quad\Phi\in\Lambda^\omega.
\end{equation}
\item Let $U\in\GL(2,\Z)$. Define \textit{base change} $U: \Lambda^\omega\to\Lambda^\omega$ by
\begin{equation}
U(\Phi)(n,m)=\Phi(n',m'), ~~\binom{n'}{m'}=U\cdot\binom{n}{m},\quad\Phi\in\Lambda^\omega.
\end{equation}
\end{enumerate}

We point out that these two operations are pairwise commutative. Besides, $U$ commute with $\Conj_B$ while $M_\lambda\circ \Conj_B=\Conj_{B(\lambda\cdot)}\circ M_\lambda.$

For $\Z^2$-action $\Phi\in\Lambda^\omega$, $n\in\N$, we define the \textit{$n$-th renormalization of $\Phi$ around $0$} to be
\begin{equation}
\cR^n(\Phi):=M_{\beta_{n-1}}\circ {Q_n}(\Phi).
\end{equation}

Based on the renormalization of $\Z^2$-action, we introduce renormalization of cocycles. For a $C^r$ cocycle $(\alpha,A)$, we consider renormlization on the related $\Z^2$-action $\Phi_{\alpha,A}$. We can also express renormalization explicitly by commuting pair:
$$\cR^n(\Phi_{\a,A})=M_{\beta_{n-1}}\circ(-1)^n\bm p_{n-1}&-p_n\\-q_{n-1}&q_n\em\bm (1,I)\\(\a,A)\em, \quad n\in\N.$$
As $\cR^n(\Phi_{\a,A})\in\Lambda_0^\omega$, by Lemma~\ref{Normalization}, there exists $B\in C^\omega(\R,\SO(s,\R))$ such that $\Conj_B\cR^n(\Phi_{\a,A})$ is normalized. We denote the commuting pair of this normalized $\Z^2$ action by $\bm(1,I)\\(\a_n,\ti A^{(n)})\em$. Since normalizing map is not unique, $\ti A^{(n)}$ is not unique either. We call $(\a_n,\ti A^{(n)})$ a \textit{representative of the $n$-th renormalization} of $(\a,A)$.

\medskip
A fundamental observation is that the conjugation relation is invariant under renormalizations. More precisely, we have the following result. The proof is relatively direct by definition, we refer to \cite[Lemma 4.4.1, Lemma 4.4.2]{P23} for example.
\begin{lem}
\label{lem.ren.conj}
1. Representative of renormalization is unique up to conjugation.\\
2. Two cocycles are conjugate to each other if and only if they have conjugate representatives of renormalization.
\end{lem}

Now we talk about the convergence of renormalization. As compact cocycle is a special case of cocycles $L^2$-conjugation to rotation, by \cite{AKP,P23}, we have the following results.
\begin{thm}
\label{sect4.0.2}
Let  $\alpha\in\R\backslash\Q$, $A\in C^\omega_h(\mathbb{T},\SO(s,\R))$. There exist $D\in\fso(s,\R)$ satisfying $\exp(D)=I$, a sequence of representatives of renormalization $(\alpha_n, \tilde{A}^{(n)})$,  $C_n\in\fz(D)$, $\ph_n\in C_h^\omega(\T,\fso(s,\R))$ such that
\begin{equation}
\label{lim.ren.thm}
\tilde{A}^{(n)}(x)=\exp(\ph_n(x))\exp((-1)^n Dx+C_n),
\end{equation}
satisfying the estimate
$$\lim_{n\to\infty}||\ph_n||_h=0.$$
\end{thm}

This means  representatives of renormalization $(\alpha_n, \tilde{A}^{(n)})$ converges to a precise normal form. This kind of normal form result was firstly stated for $\SU(2)$-cocycles in the $C^{\infty}$-topology \cite{K01,Fra04}, and then for $\SL(2,\mathbb{R})$-cocycles in the $C^\infty$ and $C^\omega$-topology \cite{AK06,AK15}. It is also generalized to quasi-periodic coycles valued in semi-compact Lie groups in the $C^\infty$-topology \cite{Ka15}.

\section{Local results I. close to constant}
To prove the conjugation result in Theorem~\ref{thm.ac.2}, by the preservation of conjugation along renormalization, it can be achieved by conjugating renormalizations. Moreover, by the convergence of renormalization given in Theorem~\ref{sect4.0.2}, we will fall into a neighborhood of some normal form after deep enough renormalization. Therefore we only need to study conjugation in local case, namely close to some normal form $(\a,\exp(Dx+C))$, where $\exp(D)=I$, $c\in\fz(D)$. We will deal with $D=0$ in this section and $D\neq0$ in the next section. 

\smallskip

In this section, we will prove that any one-frequency $\SO(3,\R)$-cocycle close to constant is $C^\omega$-almost-reducible. 

We start with introducing some notations. Let $h>0$, $d\in\N^*$. We define the strip
$$\T_h^d:=\{\theta=(\theta_1,\cdots,\theta_d)\in \C^d\, \big|\, |\Im \theta_1|+\cdots +|\Im \theta_d|\leq h\}/\Z^d.$$
We denote by $C_h^\omega(\T^d,*)$ the space of all maps $F:\T^d\to*$ admitting an analytic extension on $\T_h^d$, equipped with the norm
$$ \|F\|_h := \sup_{\theta\in \T^d_h}\|F(\theta)\|,$$ 
where $*$ could be $\R$, $\C$, $\SO(3,\R)$ or $\fso(3,\R)$.
Then $C^\omega(\T^d,*):=\bigcup_{h>0}C_h^\omega(\T^d,*)$ becomes the set of all analytic  $*$-valued function on $\T^d$. 

Our main result in this section is the following.

\begin{thm}
\label{thm.loc.dis}
Let $\a\in\R\backslash \Q$, $h>0$, $L\geq 1$,  $C\in\fso(3,\R)$.    There exists $\delta=\delta(h,L)$, such that for any $A\in C_h^\omega(\T,\SO(3,\R))$,
if  $||A- \exp(C)||_h<\delta$, then $(\a,A)$ is  $C^\omega$-almost-reducible, in the sense that,  there exist $h_n>0$,  $B_n\in C_{h_n}^\omega(\T,\SO(3,\R))$, $n\in\N^*$, such that
\be\label{est.-thm.loc.dis}
\lim_{n\rightarrow \infty} \|A_n-\exp(C_n)\|_{h_n}\|B_n\|_{h_n}^L=0,
\ee
 for $A_n:= B_n(\cdot+\a)AB_n^{-1}$ and some $C_n\in  \fso(3,\R)$, $n\in\N^*$.
\end{thm}

\begin{rem}
The perturbation  $\delta=\delta(h)$ does not depend on the frequency $\a$. Thus the result is of semi-local nature \cite{FK}.
\end{rem}

For technical reasons, we will not prove Theorem \ref{thm.loc.dis} directly. Instead we consider almost reducibility of linear systems. We will firstly formulate a parallel theorem to Theorem~\ref{thm.loc.dis} for linear systems and show that one can deduce Theorem~\ref{thm.loc.dis} from the parallel result in section~\ref{change.to.lin.s.sect}. Then we will work on almost reducibility of linear systems in the rest of this section.

\subsection{Change to almost reducibility of linear systems}
\label{change.to.lin.s.sect}
Let $\omega\in\R^2$, $A\in C^\omega(\T^2,\SO(3,\R))$. Consider the linear system
\be
\label{lin-sys}
\begin{cases}
    \dot{x}=A(\theta)x\\
    \dot{\theta}=\omega.
    \end{cases}
\ee
We say that $B\in C^\omega(\T^2, \SO(3,\R))$  conjugates (\ref{lin-sys}) to another linear system
\be
\label{lin-sys-conj}
\begin{cases}
    \dot{x}=A_+(\theta)x\\
    \dot{\theta}=\omega,
    \end{cases}
\ee
for some $A_+\in C^\omega(\T^2,\SO(3,\R))$, if
$$\partial_{\omega} B(\theta) =B(\theta)A(\theta)-A_+(\theta)B(\theta),$$
where 
$$\partial_{\omega} B(\theta_1,\theta_2):= \omega_1\frac{\partial}{\partial\theta_1}B(\theta_1,\theta_2)+\omega_2\frac{\partial}{\partial\theta_2}B(\theta_1,\theta_2),\quad \text{for }\omega=(\omega_1,\omega_2).$$

Denote by $\Phi\in C^\omega(\R\times\T^2, \SO(3,\R))$ the basic solution (which is in fact unique) of   (\ref{lin-sys}) that  satisfies
\begin{eqnarray*}
\frac{d}{dt}\Phi(t,\theta)=A(\theta+t\omega)\Phi(t,\theta), \quad \Phi(0,\theta)=I, \quad t\in\R,\, \theta\in \T^2.
\end{eqnarray*}
Let $\Phi_+$ be the basic solution of (\ref{lin-sys-conj}). Then $B$  conjugating (\ref{lin-sys}) to \eqref{lin-sys-conj} is equivalent to
\be
B(\theta+t\omega)\Phi(t,\theta)B(0,\theta)^{-1}=\Phi_+(t,\theta).
\ee

We claim the following almost reducibility result for linear system, whose proof will be given later.

\begin{thm}
\label{thm.loc.cont.1}
Let $\a\in(0,1)\backslash \Q$, $\omega=(\a,1)\in\R^2$, $C\in\fso(3,\R)$. For $h>0$ and $L\geq 1$, there exists $\delta=\delta(h,L)$  such that for any  $F\in C^\omega_h(\T^2,\fso(3,\R))$
 satisfying $||F||_h<\delta$, the system
\be
\label{loc.cont.1}
\begin{cases}
    \dot{x}=(C+F(\theta))x\\
    \dot{\theta}=\omega
    \end{cases}
\ee
is $C^\omega$-almost-reducible, in the sense that,  there exist $h_n>0$,  $B_n\in C_{h_n}^\omega(\T^2,\SO(3,\R))$, $n\in\N^*$, such that $B_n$ conjugates
(\ref{loc.cont.1}) to
\be\label{loc.cont.1-1} \begin{cases}
    \dot{x}=(C_n+F_n(\theta))x\\
    \dot{\theta}=\omega,
    \end{cases}\ee
for some $C_n\in\fso(3,\R), F_n\in C^\omega_{h_n}(\T^2,\fso(3,\R))$, $n\in\N^*$, and the following estimate holds
\be\label{est.-thm.loc.cont.1}
\lim_{n\rightarrow\infty} \|F_n\|_{h_n} \|B_n\|_{h_n}^L=0.
\ee
\end{thm}

To show that we can deduce Theorem~\ref{thm.loc.dis} from Theorem \ref{thm.loc.cont.1} directly, we need to use in addition the following local embedding theorem proved in \cite{YZ13}. It was proved for any Lie group $G$, and we will restrict on the case $\SO(3,\R)$.

\begin{thm}[\cite{YZ13}, Local embedding theorem]
\label{localemb}
Let $h>0$, $d\in\N^*$, $\mu\in \T^{d-1}$ with  $(\mu,1)$ rationally
independent,  $C\in \mathfrak{so}(3,\R)$, $H\in
C^\omega_{h}(\T^{d-1},\mathfrak{so}(3,\R))$. There exist
 $\epsilon=\epsilon(C,h,|\mu|)>0$,
$c=c(C,h,|\mu|)>0$ such that the quasi-periodic cocycle $(\mu,e^C
e^{H(\cdot)})$ can be analytically embedded into a quasi-periodic
linear system provided that $\|H\|_h<\epsilon$. More precisely,
 there
exists $F\in
C^\omega_{h/(1+|\mu|)}(\T^d,\mathfrak{so}(3,\R))$ such that $(\mu,e^C e^{H(\cdot)})$ is the Poincar\'{e} map of
\begin{eqnarray*}
\left\{ \begin{array}{l}\dot{x}=(C+F(\theta))x \\
\dot{\theta}=(\mu,1),
\end{array} \right.
\end{eqnarray*}
and the estimate holds
$$\|F\|_{h/(1+|\mu|)}\leq c \|H\|_h.$$
\end{thm}

\begin{proof}[Proof of Theorem~\ref{thm.loc.dis}]
 Assume that $\a\in(0,1)$ without lose of generality. For $A\in C_h^\omega(\T,\SO(3,\R))$ with $||A-\exp(C)||_h$ small for some constant $C\in\fso(3,\R)$, by Theorem \ref{localemb}, there exists $F\in C^\omega_{h/(1+\a)}(\T^2,\mathfrak{so}(3,\R))$ such that $(\a, A)$ is the Poincar\'{e} map of the linear system
\begin{eqnarray}
\label{change.to.lin.s}
\left\{ \begin{array}{l}\dot{x}=(C+F(\theta))x \\
\dot{\theta}=(\a,1),
\end{array} \right.
\end{eqnarray}
and the estimate holds
$$\|F\|_{h/(1+\a)}\leq c ||A-\exp(C)||_h.$$
As long as $||A-\exp(C)||_h$ is small enough, by Theorem \ref{thm.loc.cont.1}, \eqref{change.to.lin.s} is $C^\omega$-almost-reducible, i.e.,  there exist $h_n>0$,  $B_n\in C_{h_n}^\omega(\T^2,\SO(3,\R))$, $n\in\N^*$, such that $B_n$ conjugates
(\ref{change.to.lin.s}) to
\be\label{loc.cont.1-1.pf} \begin{cases}
    \dot{x}=(C_n+F_n(\theta))x\\
    \dot{\theta}=\omega,
    \end{cases}\ee
for some $C_n\in\fso(3,\R), F_n\in \cB_{h_n}$, $n\in\N^*$, and the following estimate holds
\be\label{est.-thm.loc.cont.1}
\lim_{n\rightarrow\infty} \|F_n\|_{h_n} \|B_n\|_{h_n}^L=0.
\ee

Concerning the quantitative estimates, suppose that $\Phi(t,\theta)$, $\Phi_n(t,\theta)$ are the corresponding
basic solution of $\eqref{change.to.lin.s}$ and  $(\ref{loc.cont.1-1.pf})$, then we have
\be
\label{al-ref3}
\Phi_n(t,\theta)=e^{C_nt}\Big(I+\int_0^t e^{-C_ns}F_n(\theta+s(\a,1)) \ti \Phi(s,\theta)ds \Big).
\ee
Let $g_n(t)=||e^{-C t}\Phi_n(t,\theta)||_{h_n}$.
We deduce from \eqref{al-ref3} that
$$
g(t) \leq 1 + \int_0^t \|F_n\|_{h_n}g(s)ds.
$$
Gronwall's inequality therefore provides
\be
\label{es.g.t}
g(t)\leq e^{\|F_n\|_{h_n} t}.
\ee

Back to (\ref{al-ref3}), let $t=1$, $\theta=(\theta_1,0)$, we have
\be
\Phi_n(1,\theta_1,0)=e^{C_n}+\widetilde{F}_n(\theta_1),
\ee
where $\widetilde{F}_n\in C_{h_n}^\omega(\T,\fso(3,\R))$ is given by
$$\ti F_n(\theta_1)=e^{C_n}\int_0^1 e^{-C_ns}F_n(\theta_1+s\a,s) \ti \Phi(s,\theta_1,0)ds.$$
Then by \eqref{es.g.t} we deduce that
\be
\label{F.n.ti.pf}
\|\widetilde{F}_n\|_{h_n}\leq\int_0^1\|F_n\|_{h_n}g(s) ds\leq e^{\|F_n\|_{h_n}}-1\leq2\|F_n\|_{h_n},
\ee
where the last inequation is provided by smallness of $\|F_n\|_{h_n}$. 

Since $B_n$ conjugates \eqref{change.to.lin.s}  to \eqref{loc.cont.1-1.pf}, we have the relation of basic solutions
$$B_n(\theta_1+t\alpha,t)\Phi(t,\theta_1,0)B_n(\theta_1,0)^{-1}=\Phi_n(t,\theta_1,0).$$ 
Let $t=1$. As $(\a,A)$ is the Poincar\'e map of \eqref{change.to.lin.s}, we deduce that
$$A_n(\theta_1)=B_n(\theta_1+\alpha,0)A(\theta_1)B_n(\theta_1,0)^{-1}=e^{C_n}+\widetilde{F}_n(\theta_1),$$
where we use the fact that $B$ is $1$-periodic in both coordinates. Besides, \eqref{F.n.ti.pf} and \eqref{est.-thm.loc.cont.1} give that 
$$||A_n-e^{C_n}||_{h_n}||B_n||_{h_n}^L=0.$$
Hence we have completed the proof of Theorem~\ref{thm.loc.dis}.
\end{proof}

\subsection{Proof of Theorem~\ref{thm.loc.cont.1}}
In this subsection, we will introduce an alternative analytic space and norm to work with. We will also introduce an iteration scheme, by applying which we will be able to prove Theorem~\ref{thm.loc.cont.1}. The proof of the iteration scheme will be given in the next subsection.

\subsubsection{Spaces and Norms}
To make the future estimates simpler, we introduce an alternative analytic space and norm.  For $h>0$, we define $\cB_h(*)$ as the space of functions $F\in C^\omega(\T^2, *)$ ($*$ can be $\fso(3,\R)$, $\SO(3,\R)$, $\R$ or $\C$) satisfying
$\sum_{k\in\Z^2}||\hat F(k)||e^{2\pi|k|h}<\infty$, where $|k|:=|k_1|+|k_2|$ for $k=(k_1,k_2)$,  equipped  with the norm
$$||F||^\#_h:= \sum_{k\in\Z^2}||\hat F(k)||e^{2\pi|k|h}<\infty. $$
It follows immediately that for any $F,G\in \cB_h(\SO(3,\R))$, we have $FG\in \cB_h(\SO(3,\R))$, and
$$||FH||^\#_h\leq ||F||^\#_h||H||^\#_h. $$
Sometimes, for an analytic  matrix valued function, we need to consider the relations between its analytic norm and the analytic  norms  of its entry functions. As for $\fso(3,\R)$ case, we obviously have the conclusion
\be
\|\sum_{s=1}^3 f_s J_s\|^\#_h=\left(\sum_{s=1}^3 (\|f_s\|^\#_h)^2 \right)^{\frac{1}{2}},\quad \forall f_1,f_2,f_3\in\cB_h(\R),\,h>0.
\ee

For any $N>0$, we define the operators $\cT_N$,  $\cR_N$ on $\cB_h(*)$ as
$$(\cT_NF)(\theta):=\sum_{k\in\Z^d,|k|<N}\hat F(k)e^{2\pi i<k,\theta>},\quad
(\cR_NF)(\theta):=\sum_{k\in\Z^d,|k|\geq N}\hat F(k)e^{2\pi i<k,\theta>}. $$
Then $\cT_N+\cR_N=I$.  And for any $0<h'<h$, since $||\hat F(k)||\leq ||F||_he^{-2\pi|k|h}$, we compute that
\begin{align}
\label{norm.1}
||\cR_NF||_{h'}^\#&=\sum_{k\in\Z^d,|k|\geq N}||\hat F(k)||e^{2\pi |k|h'}\leq\frac{2||F||_h}{\min\{1,(h-h')^2\}}e^{-2\pi N(h-h')}.
\end{align}

The spaces $\cB_h(*)$ and $ C^\omega(\T^2, *)$ are closely related. For $h>0$, $F\in \cB_{h}(*)$, we have
\begin{align}
\sup_{x\in\T^2_h}||F(x)||\leq||F||^\#_h.
\end{align}
Hence $F\in C_h^\omega(\T^2,*)$. It follows that
$$\cB_{h}(*)\subseteq  C_h^\omega(\T^2,*).$$
And for any $h>h'>0$, $F\in C_h^\omega(\T^2,*)$, we have the estimate
\be
\label{norm.2}
\sum_{k\in\Z^2}||\hat F(k)|| e^{2\pi |k|h'}
\leq\sum_{k\in\Z^2}||F||_h e^{-2\pi |k|h}e^{2\pi|k|h'}
\leq \frac{2||F||_h}{\min\{1,(h-h')^2\}}.
\ee
Therefore $F\in \cB_{h'}(*)$. We deduce that $$ C_h^\omega(\T^2,*)\subseteq \cB_{h'}(*).$$

\subsubsection{Proof of local almost reducibility of linear systems}

We will prove Theorem~\ref{thm.loc.cont.1} via a nonstandard KAM iterations. The key iteration lemma is the following. 
\begin{prop}
\label{prop.scheme}
Let $0<h\leq1$, $C\in\fso(3,\R)$, $F\in \cB_h(\fso(3,\R))$, $\e=||F||_h^\#$, $p\leq q<q_+$ be positive integers.
Under the following assumptions
\be
\label{cond.1}
e^{-\frac{1}{2}q_+h}<  \e \leq \min\{(10L)^{-100L},\, h^{4(L+1)},\,e^{-\frac{1}{2}q h}\},
\ee
\be
\label{cond.2}
\frac{1}{2q_+}\leq |q\alpha-p|\leq \frac{1}{q_+},
\ee
\be
\label{cond.3}
\{k\in\Z^2~|\, \langle k,\,\omega \rangle|<\frac{1}{7q}, \, |k|<\frac{q_+}{6}\}\subseteq\{l(q,-p)~|~l\in\Z\},
\ee
one can construct  $B\in \cB_{h_+}(\SO(3,\R))$, $C_+\in\fso(3,\R)$ and $F_+\in \cB_{h_+}(\fso(3,\R))$, with $h_+= \frac{h}{6(L+2)} $, such that $B$ conjugates
$\begin{cases}
\dot{x}=(C+F(\theta))x\\
\dot{\theta}=\omega
\end{cases}
$
to
$~\begin{cases}
\dot{x}=(C_+ + F_+(\theta))x\\
\dot{\theta}=\omega
\end{cases}$
, and for $\e_+=||F_+||_{h_+}^\#$, the estimate holds
\begin{eqnarray}
\label{est.-prop.scheme}
\e_+\left(||B||^\#_{h_+}\right)^{2L}\leq \e^2 e^{-q_+h_+}.
\end{eqnarray}
\end{prop}

We point out that \eqref{cond.3} requires the resonant sites to be in some special form. In fact this is satisfied as a consequence of following small divisor lemma given in \cite[Lemma 4.1]{HY12}.
\begin{lem}[\cite{HY12}]
\label{q-Lem}
Let $\a\in\R\backslash\Q$, $\omega=(\a,1)$, $\frac{p_n}{q_n}$ be the best approximation of $\a$, we have
\be
\{k\in\Z^2~|\, \langle k,\,\omega \rangle|<\frac{1}{7 q_n}, \, |k|<\frac{q_{n+1}}{6}\}\subseteq\{l(q_n,-p_n)~|~l\in\Z\}.
\ee
\end{lem}

The proof of Proposition~\ref{prop.scheme} is rather technical and will be given in the next subsection. Now we use the proposition to prove Theorem~\ref{thm.loc.cont.1}.

\begin{proof}[Proof of Theorem~\ref{thm.loc.cont.1}]
We choose 
\be
\label{def.del.pf}
\delta=\delta(h,L)=\frac{h^2}{8}\min\{(10L)^{-100L}, (\frac{h}{2})^{4(L+1)}, e^{-\frac{h}{4}}\}.
\ee
For $F\in C_h^\omega(\T^2,\fso(3,\R))$ with $||F||_h<\delta$, by \eqref{norm.2}, we deduce that $F\in\cB_{\frac{h}{2}}(\fso(3,\R))$ satisfies the estimate
\be
\label{bd.F.well.pf}
||F||_{\frac{h}{2}}^\#\leq\frac{8}{h^2}||F||_h<\frac{8}{h^2}\delta.
\ee

To study almost reducibility of the linear system
\be
\label{lin.sys.pf.0}
\begin{cases}
    \dot{x}=(C+F(\theta))x\\
    \dot{\theta}=\omega,
    \end{cases}
\ee
we construct conjugation inductively. 

\smallskip

Let $B_0= I$, $h_0=\frac{h}{2}$, $C_0= C$, $F_0= F$. Then $B_0$ automatically conjugates \eqref{lin.sys.pf.0}
 to 
$$\begin{cases}
    \dot{x}=(C_{0}+F_{0}(\theta))x\\
    \dot{\theta}=\omega,
    \end{cases}$$
and the estimate holds for $\e_0=\|F_0\|_{h_0}^\#$,
\begin{eqnarray}
\e_0\left(||B_0||^\#_{h_0}\right)^L =\e_0\leq  \min\{(10L)^{-100L}, h_0^{4(L+1)}, e^{-\frac{1}{2}q_0h_0}\},
\end{eqnarray}
where we use \eqref{def.del.pf} and \eqref{bd.F.well.pf}.

To do induction, we assume that for some $n\in\N^*$, there exist $h_n>0$, $B_n\in \cB_{h_n}(\SO(3,\R))$, $C_n\in\fso(3,\R)$ and $F_n\in \cB_{h_n}(\fso(3,\R))$, such that $B_n\in \cB_{h_n}(\SO(3,\R))$ conjugates \eqref{lin.sys.pf.0} to 
\be
\label{sys-KAM-n}
\begin{cases}
    \dot{x}=(C_{n}+F_{n}(\theta))x\\
    \dot{\theta}=\omega,
    \end{cases}
\ee
and the estimate holds for $\e_n=\|F_{n}\|_{h_{n}}^\#$,
\begin{eqnarray}
\label{est-KAM-n}
\e_n \left(||B_n||^\#_{h_{n}}\right)^L\leq  \min\{(10L)^{-100L},\,(h_{n})^{4(L+1)},\,e^{-\frac{1}{2}q_{n}h_{n}}\}.
\end{eqnarray}

\smallskip

To construct conjugation at step $n+1$, we separate into two different cases, depending on the size of $\e_n$.

{\it Case 1.} If $\e_n\leq e^{-\frac{1}{2}q_{n+1}h_n}$, then we trivially define
$$
h_{n+1}=h_n,\quad B_{n+1}=B_n,\quad C_{n+1}=C_n,\quad F_{n+1}=F_n.
$$
Immediately,  $B_{n+1}$
conjugates \eqref{lin.sys.pf.0} to 
$$\begin{cases}
    \dot{x}=(C_{n+1}+F_{n+1}(\theta))x\\
    \dot{\theta}=\omega,
    \end{cases}$$
 and the estimate holds for $\e_{n+1}=\|F_{n+1}\|_{h_{n+1}}^\#$
\begin{eqnarray*}
\e_{n+1} \left(||B^{(n+1)}||^\#_{h_{n+1}}\right)^L
 \leq \min\{(10L)^{-100L},\,(h_{n+1})^{4(L+1)},\,e^{-\frac{1}{2}q_{n+1}h_{n+1}}\}.
\end{eqnarray*}

\smallskip

{\it Case 2.} If $ \e_n> e^{-\frac{1}{2}q_{n+1}h_n}$, then together with \eqref{est-KAM-n} we have
\begin{eqnarray}
\label{smallness-cond.-case.B}
e^{-\frac{1}{2}q_{n+1}h_n} <\e_n\leq\min\{(10L)^{-100L},\,(h_n)^{4(L+1)},\,e^{-\frac{1}{2}q_nh_n}\}.
\end{eqnarray}
By preceding estimate and Lemma \ref{q-Lem}, we are able to apply Proposition~\ref{prop.scheme}. This provides the existence of $B^{(n+1)} \in \cB_{h_{n+1}}(\SO(3,\R))$, $C_{n+1}\in\fso(3,\R)$ and $F_{n+1}\in \cB_{h_{n+1}}(\fso(3,\R))$, such that $B^{(n+1)}$ conjugates (\ref{sys-KAM-n}) to
\be
\label{sys-KAM-n+1}
\begin{cases}
    \dot{x}=(C_{n+1}+F_{n+1}(\theta))x\\
    \dot{\theta}=\omega,
    \end{cases}
\ee
and the estimate holds
\be
\label{es.indu.add.2}
\e_{n+1}\left(||B^{(n+1)}||^\#_{h_{n+1}}\right)^{2L}\leq \e_n^2 e^{-q_{n+1}h_{n+1}}.
\ee
In particular,
\be
\label{es.indu.add.3}
\e_{n+1}\leq\e_n^2.
\ee
Then $B_{n+1}:=B^{(n+1)} B_{n}$ conjugates  \eqref{lin.sys.pf.0} to (\ref{sys-KAM-n+1}). Moreover, by \eqref{est-KAM-n},\eqref{es.indu.add.2} and \eqref{es.indu.add.3}, we have the following estimate
\begin{align*}
\e_{n+1} \left(||B_{n+1}||^\#_{h_{n+1}}\right)^L\leq &
\e_{n+1} \left(||B^{(n+1)}||^\#_{h_{n+1}}\right)^L\left(||B_n||^\#_{h_{n+1}}\right)^L \nonumber\\
\leq & \left(\e_{n+1}\left(||B^{(n+1)}||^\#_{h_{n+1}}\right)^{2L}\right)^{\frac{1}{2}} \cdot \e_n\left(||B_n||^\#_{h_{n+1}}\right)^{L}\nonumber\\
\leq & \left( \min\{(10L)^{-100L},\, (h_{n})^{4(L+1)},\,e^{-\frac{1}{2}q_{n}h_{n}}\} \right)^2 \cdot e^{-\frac{1}{2}q_{n+1}h_{n+1}}\nonumber\\
\leq & \min\{(10L)^{-100L},\, (h_{n+1})^{4(L+1)},\,e^{-\frac{1}{2}q_{n+1}h_{n+1}}\}.
\end{align*}

Hence in both case, we are able to construct conjugacy $B_n\in\cB_{h_{n+1}}(\SO(3,\R))$ and satisfy estimates in the same form. 

\smallskip
By induction, we conclude that for any $n\in\N$, there exist $h_n$, $B_n\in \cB_{h_n}(\SO(3,\R))$, $C_{n}\in\fso(3,\R)$, $F_{n}\in \cB_{h_n}(\fso(3,\R))$, such that $B_n$ conjugates   \eqref{lin.sys.pf.0} to
$$
\begin{cases}
    \dot{x}=(C_{n}+F_{n}(\theta))x,\\
    \dot{\theta}=\omega,
    \end{cases}
$$
and the estimate holds
\begin{eqnarray}\label{conclude.n}
\|F_{n}\|_{h_{n}}^\# \left(||B_n||^\#_{h_{n}}\right)^L
 \leq \min\{(10L)^{-100L},\,(h_{n})^{4(L+1)},\,e^{-\frac{1}{2}q_{n}h_{n}}\}.
\end{eqnarray}
Moreover, \eqref{conclude.n} provides that
$$\lim_{n\to\infty}\|F_{n}\|_{h_{n}}^\# \left(||B_n||^\#_{h_{n}}\right)^L=0.$$
As shown in \eqref{norm.1}, $\cB_{h_n}(*)\subset C_h^\omega(\T^2,*)$ with $||\cdot||_h\leq||\cdot||_h^\#$, we deduce that $B_n\in C^\omega_{h_n}(\T^2,\SO(3,\R))$, $F_{n}\in C^\omega_{h_n}(\T^2,\fso(3,\R))$ and 
$$\lim_{n\to\infty}\|F_{n}\|_{h_{n}} \left(||B_n||_{h_{n}}\right)^L=0.$$
Hence we complete the proof of Theorem~\ref{thm.loc.cont.1}.
\end{proof}

\subsection{Proof of iteration scheme}
In this subsection, we construct the iteration scheme. We will firstly introduce a nice subspace of $\cB_h(\fso(3,\R))$ in the sense of conjugation, then give the proof of Proposition~\ref{prop.scheme}.

For simplicity of notation, we denote $\cB_h(\fso(3,\R))$ briefly by $\cB_h$. Without lose of generality, we assume that
$$ C=2\pi \rho J_1,\qquad \rho\in\R.$$

\subsubsection{A nice subspace of $\cB_h$}
For $F\in\cB_h$, to construct $B\in \cB_h(\SO(3,\R))$ so that for some $C_+\in\fso(3,\R)$ and a smaller $F_+\in\cB_h$, $B$ conjugating $\begin{cases}\dot{x}=(C+F(\theta))x\\\dot{\theta}=\omega\end{cases}$ to $\quad\begin{cases}\dot{x}=(C_++F_+(\theta))x\\\dot{\theta}=\omega\end{cases}$, it is sufficient to solve
\be\label{gener-co-hom-eq}
\partial_{\omega }B=B(C+F)-(C_+ + F_+)B.
\ee
If  $B$ is close to the identity, namely $B=e^Y$ with $Y\in \cB_h$ small, then (\ref{gener-co-hom-eq}) provides a linearized equation
\be\label{lin-co-hom-eq}
\partial_{\omega }Y-[C,Y]=F+D,
\ee
for some constant $D\in \fso(3,\R)$. However, when solving (\ref{lin-co-hom-eq}), one may encounter the small divisor problem, i.e., the liner operator
\be \label{lin-op}
\cB_h\rightarrow\cB_h,\quad Y\mapsto \partial_{\omega }Y-[C,Y],
\ee
may be not invertible or invertible but with an inverse operator of too large norm.

As studied in \cite{HY12, HSY19}, a key observation is that there exists a subspace $\cB_h^{(nre)}$ of $\cB_h$, satisfying the property that for any $Y\in \cB_h^{(nre)}$, one has
\be
\label{prop.B.nre.1}
\partial_{\omega}Y,\, [C,Y]\in\cB_h^{(nre)},
\ee
and the estimate holds
\be
\label{prop.B.nre.2}
||\partial_{\omega} Y-[C,Y]||_h^\#\geq\eta||Y||^\#_h,
\ee
for a uniform constant $\eta>0$. In other word, the restricted linear operator \eqref{lin-op} on $\cB_h^{(nre)}$ behaves well. Therefore, we are able to solve \eqref{lin-co-hom-eq} in $\cB_h^{(nre)}$. To formulate the result precisely, we decompose the space
$\cB_h$ into a direct sum $\cB_h^{(nre)}\oplus\cB_h^{(re)}$. 

\begin{lem}[\cite{HY12, HSY19}]
\label{one-step-KAM}
Let $h>0$, $\eta>0$, $\omega\in\R^2$, $C\in\fso(3,\R)$. For any $F\in \cB_h$ satisfying $||F||^\#_h\leq 10^{-8} \eta^2$, there exist $Y\in\cB_h$, $H\in\cB_h^{(re)}$ such that
 $e^Y$ conjugate
$\begin{cases}
\dot{x}=(C+F(\theta))x\\
\dot{\theta}=\omega
\end{cases}$
to \quad  $\begin{cases}
\dot{x}=(C + H(\theta))x\\
\dot{\theta}=\omega
\end{cases}$,
satisfying the estimates
\be
||Y||^\#_h \leq \frac{10}{\eta}||F||^\#_h,\qquad||H||^\#_h \leq 2 ||F||^\#_h .
\ee
\end{lem}

\begin{rem}
Preceding lemma was firstly given in \cite{HY12} and there are some minor improvements in later articles such as \cite{HSY19}.
Here we use the version in \cite{HSY19}.
\end{rem}

Given  $\e>0$. We define
\begin{equation*}
\Lambda_1:=\{k\in\Z^2\,|\, |\langle k,\, \omega\rangle|\geq 2 \e^{\frac{1}{4}}\},\quad
\Lambda_2:=\{k\in\Z^2\,|\, |\langle k,\, \omega\rangle-\rho|\geq \e^{\frac{1}{4}}\}.
\end{equation*}
Then it follows directly that
\be\label{Lambda-relation}
\Lambda_2^c-k_*:= \{k-k_*\,|\,k\in \Lambda_2^c\,\}\subseteq \Lambda_1^c,\quad \forall k_*\in \Lambda_2^c.
\ee
Furthermore, we have the  following simple fact.

\begin{lem}\label{small-div-Lem}
Under the assumptions of Proposition~\ref{prop.scheme}, the estimate holds
\be \label{est.1-small-div-Lem}
100(L+2)^2q \e^{\frac{1}{4}}\ll 1.
\ee
Consequently, we have
\be \label{est.2-small-div-Lem}
\Lambda_1^c \cap\{k\in\Z^2~|~|k|<\frac{q_+}{6}\}\subseteq \{l(q,-p)~|~l\in\Z\}.
\ee
\end{lem}
\begin{proof}
By (\ref{cond.1}), we get
\begin{eqnarray*}
100(L+2)^2q \e^{\frac{1}{4}}&\leq &100 (L+2)^2 q \e^{\frac{1}{8}} \e^{\frac{1}{8}}\leq 100 (L+2)^2 q e^{-\frac{qh}{16}} \e^{\frac{1}{16}}\\
&\leq &\frac{1600
(L+2)^2}{e h}\e^{\frac{1}{8}}\leq 10^5 L^2\e^{\frac{1}{8}}\ll 1,
\end{eqnarray*}
then \eqref{est.2-small-div-Lem} follows from \eqref{cond.3} and  \eqref{est.1-small-div-Lem}.
\end{proof}

Given $h>0$, now we define $\cB_h^{(nre)}$, $\cB_h^{(re)}$ explicitly and verify the properties \eqref{prop.B.nre.1} and \eqref{prop.B.nre.2}. Let
\begin{align*}
\cB_h^{(nre)}:= \{E\in \cB_h\,|\, &E=fJ_1+ \Re wJ_2+ \Im wJ_3,\, \text{ for }f\in\cB_h(\R)\text{ in the form }\sum_{k\in \Lambda_1}
\hat{f}(k)e^{2\pi i \langle k,\theta\rangle},\nonumber\\
&w\in\cB_h(\C) \text{ in the form }\sum_{k\in \Lambda_2}
\hat{w}(k)e^{2\pi i \langle k,\theta\rangle}\},
\end{align*}
\begin{align}
\label{def.B.h.re}
\cB_h^{(re)}:= \{H\in \cB_h\,|\, &H=fJ_1+ \Re wJ_2+ \Im wJ_3,\,  \text{ for }f\in\cB_h(\R)\text{ in the form }\sum_{k\in \Lambda_1^c}
\hat{f}_1(k)e^{2\pi i \langle k,\theta\rangle},\nonumber\\
& w\in\cB_h(\C) \text{ in the form }\sum_{k\in \Lambda_2^c}
\hat{w}(k)e^{2\pi i \langle k,\theta\rangle}\}.
\end{align}
It follows the direct sum splitting
\begin{equation*}
 \cB_h=\cB_h^{(nre)}\oplus\cB_h^{(re)}. 
\end{equation*}
By definition, for any $E\in \cB_h^{(nre)}$, there exist $f=\sum_{k\in \Lambda_1}
\hat{f}(k)e^{2\pi i \langle k,\theta\rangle}\in \cB_h(\R)$,  $w=\sum_{k\in \Lambda_2}
\hat{w}(k)e^{2\pi i \langle k,\theta\rangle}\in \cB_h
(\C)$, such that $E=fJ_1+ \Re wJ_2+ \Im wJ_3$. We can express $E$ explicitly 
\begin{align*}
 E=&\left(\sum_{k\in \Lambda_1}\hat{f}(k)e^{2\pi i \langle k,\theta\rangle}\right)J_1\\
&+\frac{1}{2}\left(\sum_{k\in \Lambda_2}\hat{w}(k)e^{2\pi i \langle k,\theta\rangle}
+\sum_{k\in \Lambda_2}\overline{\hat{w}(k)}e^{-2\pi i \langle k,\theta\rangle}\right)J_2\\
&-\frac{i}{2}  \left(\sum_{k\in \Lambda_2}\hat{w}(k)e^{2\pi i \langle k,\theta\rangle}
-\sum_{k\in \Lambda_2}\overline{\hat{w}(k)}e^{-2\pi i \langle k,\theta\rangle}\right)J_3.
\end{align*}
Then we compute that
\begin{align*}
\partial_{\omega} E=&2\pi i \left(\sum_{k\in \Lambda_1}\langle k,\, \omega\rangle\hat{f}(k)e^{2\pi i \langle k,\theta\rangle}\right)J_1\\
&+\pi i \left(\sum_{k\in \Lambda_2}\langle k,\, \omega\rangle\hat{w}(k)e^{2\pi i \langle k,\theta\rangle}
-\sum_{k\in \Lambda_2}\langle k,\, \omega\rangle\overline{\hat{w}(k)}e^{-2\pi i \langle k,\theta\rangle}\right)J_2\\
&+\pi  \left(\sum_{k\in \Lambda_2}\langle k,\, \omega\rangle\hat{w}(k)e^{2\pi i \langle k,\theta\rangle}
+\sum_{k\in \Lambda_2}\langle k,\, \omega\rangle\overline{\hat{w}(k)}e^{-2\pi i \langle k,\theta\rangle}\right)J_3,
\end{align*}
 and
\begin{align*}
[C,E]=&\pi i \rho   \left(\sum_{k\in \Lambda_2}\hat{w}(k)e^{2\pi i \langle k,\theta\rangle}
-\sum_{k\in \Lambda_2}\overline{\hat{w}(k)}e^{-2\pi i \langle k,\theta\rangle}\right)J_2\\
&+\pi \rho \left(\sum_{k\in \Lambda_2}\hat{w}(k)e^{2\pi i \langle k,\theta\rangle}
+\sum_{k\in \Lambda_2}\overline{\hat{w}(k)}e^{-2\pi i \langle k,\theta\rangle}\right)J_3,
\end{align*}
where we recall that $C=2\pi\rho J_1$, $\rho\in\R$. It follows that
$$
 \partial_{\omega} E\in \cB_h^{(nre)},\quad  [C,E]\in \cB_h^{(nre)}.
$$
Moreover, by direct computation, we have
\be\label{non-res-div}
\|\partial_{\omega} E-[C,E]\|^\#_h\geq \pi \epsilon^{\frac{1}{4}} \|E\|^\#_h.
\ee
Hence the two characterizing conditions are satisfied therefore the space $\cB_h^{(nre)}$ is well-defined.

\subsubsection{Proof of Proposition~\ref{prop.scheme}}
Based on the nice space $\cB^{(nre)}$ defined in the preceding section, we can now prove Proposition~\ref{prop.scheme}. We will also make use of Floquet theory.

\begin{proof}[Proof of Proposition~\ref{prop.scheme}]
We start with the linear system
\be
\label{start.lin.sys.0}
\begin{cases}
\dot{x}=(C+F(\theta))x\\
\dot{\theta}=\omega,
\end{cases}
\ee
satisfying the assumptions \eqref{cond.1}, \eqref{cond.2}, \eqref{cond.3}. Motivated by the quantitative estimate \eqref{non-res-div} of $\cB^{(nre)}$, we work with $\eta=\pi\e^{\frac{1}{4}}.$
Then \eqref{cond.1} provides that
$$||F||_h^\#=\e<10^{-8}\eta,$$
which enables us to apply Lemma \ref{one-step-KAM}. It shows the existence of $Y\in \cB_h$, $H\in \cB_h^{(re)}$, such that $e^Y$ conjugates (\ref{start.lin.sys.0}) to
\be
\label{sys-0-pf}
\begin{cases}
\dot{x}=(C+H(\theta))x\\
\dot{\theta}=\omega,
\end{cases}
\ee
and the estimates hold
\be
\label{es.Y.H.pf}
||Y||^\#_h\leq \frac{10}{\eta}\e< \e^{\frac{1}{2}},\quad ||H||^\#_h\leq 2\e,
\ee
where \eqref{cond.1} is used for the first inequation. By construction of $\cB_h^{(re)}$ in  \eqref{def.B.h.re}, there exist 
\be\label{G-f-w}
  f=\sum_{k\in \Lambda_1^c}
\hat{f}(k)e^{2\pi i \langle k,\theta\rangle}\in \cB_h(\R),
 \quad w=\sum_{k\in \Lambda_2^c}
\hat{w}(k)e^{2\pi i \langle k,\theta\rangle}\in\cB_h(\C),
\ee
such that
$$H=fJ_1+\Re w J_2+ \Im w J_3.$$

\medskip
To study further conjugation of \eqref{sys-0-pf}, we separate into two cases depending on $\Lambda_2^c$.

{\it Case 1}. If $\Lambda_2^c\cap\{k\in\Z^2||k|<\frac{q_+}{6}\}=\emptyset$, then $\mathcal{T}_{\frac{q_+}{6}}
H=\mathcal{T}_{\frac{q_+}{6}}f J_1$. By Lemma~\ref{small-div-Lem},
$$
\Lambda_1^c \cap\{k\in\Z^2~|~|k|<\frac{q_+}{6}\}\subseteq \{l(q,-p)~|~l\in\Z\}.
$$
Correspondingly, we define
\begin{align}\label{B-form-case1}
B_1(\theta):= &\exp\left(\sum_{k\in\Lambda_1^c\backslash \{0\},~|k|<\frac{q_+}{6}}\frac{\hat{f}(k)e^{2\pi i\langle k,\theta \rangle}}{2\pi i\langle k,\,\omega\rangle}J_1\right)\nonumber\\=&
\exp\left(\sum_{l(q,-p)\in\Lambda_1^c\backslash \{0\},|l(q,-p)|<\frac{q_+}{6}}\frac{\hat{f}(k)e^{2\pi i \langle l(q,-p),\theta \rangle}}{2\pi i l (q\alpha-p)}J_1 \right).
\end{align}
Immediately, $B_1$ conjugates (\ref{sys-0-pf}) to
\begin{equation}
\label{sys.1.end}
\begin{cases}
\dot{x}=(C_++ F_+)x\\
\dot{\theta}=\omega,
\end{cases}
\end{equation}
where $C_+= C+\hat{H}(0)$, $F_+= \Ad(B_1)\mathcal{R}_{\frac{q_+}{6}} H$.
By the assumptions \eqref{cond.1} and \eqref{cond.2}, for $h_+:= \frac{h}{6(L+2)}$, the construction of $B_1$ \eqref{B-form-case1} provides the bound
\begin{eqnarray}\label{B-est.-case1}
 \|B_1\|^\#_{h_+} \leq e^{\frac{q_+\epsilon}{\pi}} \ll  e^{\frac{1}{2(L+1)} q_+ h_+}.
\end{eqnarray}

Let $B=B_1e^Y$. Then $B$ conjugates \eqref{start.lin.sys.0} to \eqref{sys.1.end}. For $\epsilon_+=\|F_+\|^\#_{h_+}$, by \eqref{B-est.-case1}, \eqref{es.Y.H.pf}, \eqref{cond.1}, we compute that
\begin{align}\label{F+-est.-case1}
\epsilon_+\left(\|B\|^\#_{h_+}\right)^{2L}
\leq  &\|\mathcal{R}_{\frac{q_+}{6}} H\|^\#_{h_+}\cdot \left(\|B_1\|^\#_{h_+}\right)^{2(L+1)}\cdot \left(\|e^Y\|^\#_{h_+}\right)^{2L}\nonumber\\
\leq & 2\epsilon
e^{-\frac{2 \pi q_+\cdot (6L+11) h_+ }{6}} \cdot  e^{ q_+ h_+}\cdot e^{2L\e^{\frac{1}{2}}}\nonumber\\
\ll&  \e^{2} e^{- q_+ h_+}.
\end{align}
Hence \eqref{est.-prop.scheme} holds, which completes the proof in Case 1.

\bigskip

{\it Case 2}. If $\Lambda_2^c\cap\{k\in\Z^2||k|<\frac{q_+}{6}\}\neq\emptyset$, then there exists $k_*\in\Lambda_2^c$,  such that
\begin{equation*}
 |k_*|=\min_{k\in\Lambda_2^c}|k|,\quad 0\leq |k_*|<\frac{q_+}{6}.
\end{equation*}
We deal with $k_*$ through a rotation step as follows.
\begin{lem}\label{rot-Lem}
Let $Q=\exp\{-2\pi \langle k_*,\, \theta\rangle J_1\}$, $C_1= 2\pi(\rho-\langle k_*,\theta \rangle )J_1$,
\be
\label{def.F.1.lem}
F_1=fJ_1 + \Re \{w e^{-2\pi \langle k_*,\,\theta\rangle} \}  J_2+ \Im \{w e^{-2\pi \langle k_*,\,\theta\rangle} \} J_3.
\ee
Then $Q$ conjugates (\ref{sys-0-pf}) to the system
\be\label{sys-1-pf-case2}
\begin{cases}
\dot{x}=(C_1+F_1(\theta))x\\
\dot{\theta}=\omega,
\end{cases}
\ee
and the estimates hold
\be\label{est.-rot-Lem}
\|C_1\|\leq 2\pi \epsilon^{\frac{1}{4}},\qquad \|F_1\|^\#_{\frac{h}{3}}\leq 2\epsilon.
\ee
Moreover, $\mathcal{T}_{\frac{q_+}{6}}F_1$ is in the form
\be\label{res-form-rot-Lem}
\mathcal{T}_{\frac{q_+}{6}}F_1=\sum_{k\in(q,-p)\Z, |k|<\frac{q_+}{6}} \hat{F_1}(k)e^{2\pi i\langle k,\theta\rangle}.
\ee
\end{lem}

\begin{proof}
Direct computation shows that  $Q$ conjugates (\ref{sys-0-pf}) to $\begin{cases}
\dot{x}=(C_1+\Ad(Q) H(\theta))x\\
\dot{\theta}=\omega,
\end{cases}$. In addition, we compute that
\begin{align}
\label{F-form-pf-rot-Lem}
\Ad(Q)H=&\Ad(Q) \cdot (fJ_1+\Re w J_2+ \Im w J_3)\nonumber\\
=& fJ_1+ \Re w Ad(Q)\cdot J_2+ \Im w Ad(Q) \cdot J_3\nonumber\\
=& fJ_1+ \{\Re w \cos 2\pi \langle k_*,\,\theta\rangle + \Im w\sin 2\pi \langle k_*,\,\theta\rangle\}J_2 \nonumber\\
&+\{\Im w\cos 2\pi \langle k_*,\,\theta\rangle-\Re w \sin 2\pi \langle k_*,\,\theta\rangle\}J_3 \nonumber\\
=& fJ_1 + \Re \{w e^{-2\pi \langle k_*,\,\theta\rangle} \}  J_2+ \Im \{w e^{-2\pi \langle k_*,\,\theta\rangle} \} J_3.
\end{align}
Hence we deduce \eqref{def.F.1.lem} for $F=\Ad(Q)H$.

Recall \eqref{G-f-w}, we can rewrite
\be\label{w-form-pf-rot-Lem}
w e^{-2\pi \langle k_*,\,\theta\rangle}=\sum_{k\in\Lambda_2^c} \hat{w}(k) e^{2\pi i \langle k-k_*,\,\theta\rangle}
=\sum_{k\in\Lambda_2^c-k_*} \hat{w}(k+k_*) e^{2\pi i \langle k,\,\theta\rangle}.
\ee
As $k_*\in\Lambda_2^c$, by (\ref{Lambda-relation}) and (\ref{est.2-small-div-Lem}), we have
$$
(\Lambda_2^c-k_*)\cap\{k\in\Z^2~|~|k|<\frac{q_+}{6}\}\subseteq \Lambda_1^c\cap\{k\in\Z^2~|~|k|<\frac{q_+}{6}\}\subseteq\{l(q,-p)~|~l\in\Z\}.
$$
Therefore we get the truncated form of $F$ \eqref{res-form-rot-Lem}.
\end{proof}

To do further conjugation of \eqref{sys-1-pf-case2}, we firstly work with the  truncated system
\be\label{trunc-sys-1-pf-case2}
\begin{cases}
\dot{x}=(C_1+\mathcal{T}_{\frac{q_+}{6}}F_1)x\\
\dot{\theta}=\omega.
\end{cases}
\ee
By the precise from of $\mathcal{T}_{\frac{q_+}{6}}F_1$ given in \eqref{res-form-rot-Lem},  we can apply Lemma \ref{F-Lem}. By (\ref{est.-rot-Lem}),
\be\label{est.-trunc-sys-1-pf-case2}
\|C_1+\mathcal{T}_{\frac{q_+}{6}}F_1\|^\#_{\frac{h}{3}}\leq 10 \epsilon^{\frac{1}{4}}.
\ee
Applying Lemma \ref{F-Lem} to the system \eqref{trunc-sys-1-pf-case2} on the domain of $|\Im \theta|\leq \frac{h}{3}$,  there exists  $B_1\in C^\omega_{\frac{h}{3}}(\T^2,\SO(3,\R))$  which conjugates (\ref{trunc-sys-1-pf-case2}) to
 $\begin{cases}
\dot{x}=C_+x\\
\dot{\theta}=\omega
\end{cases}$ for some $C_+\in \fso(3,\R)$.  By (\ref{est.-trunc-sys-1-pf-case2}), Lemma \ref{F-Lem} and Lemma \ref{small-div-Lem}, we have the estimate
$$
\|B_1\|_{\frac{h}{3}}\leq e^{15 q_+ q \epsilon^{\frac{1}{4}}h}\leq e^{ \frac{1}{(L+1)}q_+ h_+}.
$$
Then by the assumptions \eqref{norm.2}, \eqref{cond.1}, we deduce the estimate for $h_+=\frac{h}{6(L+2)}$, 
\begin{eqnarray}\label{B-est.1-pf-using-F-Lem}
\|B_1\|^\#_{h_+}\leq  \frac{72(L+2)^2}{(2L+3)^2 h^2}
e^{ \frac{1}{(L+1)}q_+ h_+}\leq 200 \e^{\frac{1}{L+1}}e^{ \frac{1}{2(L+1)}q_+ h_+}
\leq e^{ \frac{1}{2(L+1)}q_+ h_+}.
\end{eqnarray}

Consequently,  $B_1$ conjugates \eqref{sys-1-pf-case2} to 
\be
\label{sys.end.case.2}
\begin{cases}
\dot{x}=(C_++F_+(\theta))x\\
\dot{\theta}=\omega,
\end{cases}
\ee
where $F_+= \Ad(B_1)\mathcal{R}_{\frac{q_+}{6}} F_1$. To sum up, let
\[ B= B_1 Q e^{Y}=B_1\exp\{-2\pi \langle k_*,\, \theta\rangle J\}e^{Y}.\]
Then $B$ conjugates  \eqref{start.lin.sys.0} to \eqref{sys.end.case.2}.

By \eqref{cond.1} and (\ref{B-est.1-pf-using-F-Lem}), recalling $|k_*|\leq \frac{1}{6}q_+$, for $\e_+=\|F_+\|^\#_{h_+}$, we compute that
\begin{align*}
\e_+\left(\|B\|^\#_{h_+}\right)^{2L} \leq & \|\mathcal{R}_{\frac{q_+}{6}}F_1\|^\#_{h_+} \e_+ \left(\|B_1\|^\#_{h_+} \right)^{2L+2}  \left(\|Q\|^\#_{h_+}\right)^{2L}  \left(\|e^{Y}\|^\#_{h_+}\right)^{2L}\nonumber\\
\leq & 2\e e^{-2\pi \cdot \frac{q_+}{6}\cdot 2(L+\frac{3}{2})h_+}\cdot  e^{ q_+h_+}\cdot \left(e^{2\pi\cdot \frac{1}{6}q_+h_+}\right)^{2L} \cdot 2^{2L}\nonumber\\
\ll&  \e^2 e^{- q_+ h_+}.
\end{align*}
Hence (\ref{est.-prop.scheme}) holds, which completes the proof in Case 2. Therefore we have completed the proof of Proposition~\ref{prop.scheme}.

\end{proof}

\section{Local results II. close to $\exp(Dx+C)$ with $D\neq0$}
\label{sect.n0}

When restricted in $3$ dimension, since the dynamical degree $D\in\fso(3,\R)$, we could always conjugate it to $2\pi dJ_1$ for some $d\in\R$. By Theorem~\ref{sect4.0.2}, $\exp D=I$. It forces $d\in\Z$.

\medskip

In this section, we study the normal form in the neighborhood of $(\a,\exp(Dx+C))$ for $D,C\in\fso(3,\R)$, where $D=2\pi dJ_1$, $d\in\Z$, $C\in\fz(D)$. By Lemma~\ref{cen}, $C$ must be in the form $cJ_1$ for some $c\in\R$. We will prove the following theorem.
\begin{thm}
\label{thm.4.1}
Let $h>0$, $\gamma>0,\sigma>0$,  $d\in\Z$. There exists $\e>0$ such that for $\a\in \DC(\gamma,\sigma)$, $A\in C_h^\omega(\T,\SO(3,\R))$, if $||A-\exp(2\pi(dx+c)J_1)||_h\leq\e$ for some $c\in\R$ and $(\a,A)$ has dynamical degree $dJ_1$, then $(\a,A)$ is $C^\omega$-conjugate to $(\a,\exp(2\pi(dx+c')J_1))$ for some $c'\in\R$.
\end{thm}

\subsection{Cohomological equation}
Let $\a\in\DC(\gamma,\sigma)$. $A$ can be written in the form
$$A(x)=\exp(2\pi(dx+c)J_1)\exp \ph(x),$$
where $\ph\in C^\omega_h(\T,\fso(3,\R))$ satisfies $||\ph||_h\leq\e$. We try to find the simplest function $\ph'\in C^\omega_{h'}(\T,\fso(3,\R))$ and a $\psi\in C^\omega_{h'}(\T,\fso(3,\R))$ for some $0<h'<h$ such that
\be
\exp \psi(x+\a)\exp(2\pi(dx+c)J_1)\exp \ph(x)\exp(- \psi(x))=\exp(2\pi(dx+c)J_1)\exp \ph(x).
\ee
If $\ph$ is small enough and if we seek a small $\psi$, the linearized equation becomes
\be
\label{co.U}
\Ad(\exp(-2\pi(dx+c)J_1))\psi(x+\a)+\ph(x)-\psi(x)=\ph'(x).
\ee

Let $\{w_1,w_2,w_3\}$ be a basis of $\fso(3,\C)$
 given by the adjoint action of $\exp(2\pi(dx+c)J_1)$, i.e.,
$$\exp(-2\pi(dx+c)J_1)w_k\exp(2\pi(dx+c)J_1)=e^{2\pi i(l_kx+\la_k)}w_k,\quad k=1,2,3,$$
for some $l_k,\la_k\in\C$. As $\exp (2\pi dJ_1)=I$ we have $l_k\in\Z$. We express $\ph$ with respect to the basis $\ph(x)=\sum_{k=1}^3\ph_k(x)w_k.$ Then \eqref{co.U} becomes
\be
e^{2\pi i(l_kx+\la_k)}\psi_k(x+\a)+\ph_k(x)-\psi_k(x)=0,\quad k=1,2,3.
\ee

By Lemma~\ref{cen}, there exits only one $k\in\{1,2,3\}$ such that $l_k=0$. We assume that $l_3=0$. Then Lemma~\ref{cen} gives that  $w_3=J_3$. Therefore $\la_3=0$. In this case, we have the following classical lemma:
\begin{lem}
\label{l=0}
Let $\a\in\DC(\gamma,\sigma)$, $h>0$, $\ph\in C_h^\omega(\T,\C)$. For $0<h'<h$, there exists $\psi\in C^\omega_{h'}(\T,\C)$, such that
$$-\ph(x)+\hat\ph(0)=\psi(x+\a)-\psi(x),$$
and
$$||\psi||_{h'}\leq \cst.||\ph||_h,$$
where the constant depends only on $\gamma$.
\end{lem}

For $k\in\{1,2\}$, $l_k\neq0$. Similar to Proposition 8.1 in \cite{K01}, we have the following lemma.
\begin{lem}
\label{ln0}
Let $\a\in\R\backslash\Q, h>0, \ph\in C_h^\omega(\T,\C), l\in\Z\backslash\{0\}, \la\in\R$.
For $0<h'<h$, there exists $\psi\in C^\omega_{h'}(\T,\C)$ such that
\be
e^{2\pi i(lx+\la)}\psi(x+\a)-\psi(x)=-\ph(x)+P_{\a,l,\la}(\ph)(x),
\ee
where
$$P_{\a,l,\la}(\ph)(x)=
\begin{cases}
\sum_{m=0}^{-l-1}e^{2\pi imx}P_{\a,l,\la}(\ph)(m),\quad\text{if }l<0;\\
\sum_{m=-1+1}^{0}e^{2\pi imx}P_{\a,l,\la}(\ph)(m),\quad\text{if }l>0,\\
\end{cases}$$
\begin{eqnarray*}P_{\a,l,\la}(\ph)(m)&=&\sum_{j=0}^\infty\hat\ph(m-jl)e^{2\pi i\{[(j+1)m-\frac{j(j-1)}{2}l]\a+(j+1)\la\}} \\
&+&\sum_{j=1}^\infty\hat\ph(m+jl)e^{-2\pi i\{[jm+\frac{j(j+1)}{2}l]\a+j\la\}},\end{eqnarray*} and the following estimates hold:
$$||P_{\a,l,\la}(\ph)||_{h'}\leq ||\varphi||_h(1+\frac{2e^{-2\pi lh}}{1-e^{-2\pi lh}})\sum_{j=1}^{|l|}e^{-2\pi k(h-h')}\leq\cst.\frac{||\ph||_h}{h},$$
$$||\psi||_{h'}\leq\frac{||\varphi||_h}{1-e^{-2\pi h}}(1+\frac{1}{1-e^{-2\pi(h-h')}})\leq\cst.\frac{||\ph||_h}{h(h-h')}. $$
\end{lem}

For $\ph\in C_h^\omega(\T,\fso(3,\R))$, we decompose it with respect to the basis $\{w_1,w_2,w_3\}$ to be $\ph=\sum_{k=1}^3\ph_kw_k$. Let $\a\in\R\backslash\Q, D\in\fso(3,\R)\backslash\{0\}$, $C\in\fz(D)$. We define
$$\cP_0\ph=\hat\ph_3(0)w_3,\quad \cP_{\a,d,c}\ph=P_{\a,l_1,\la_1}\ph_1w_1+P_{\a,l_2,\la_2}\ph_2w_2.$$

Combining Lemma~\ref{l=0} and \ref{ln0} in each direction, and using preceding lemma, we deduce the following result.
\begin{cor}
\label{cor.conj}
Let $\e>0$ small, $\a\in \DC(\gamma,\sigma)$, $h>0$, $\ph\in C_h^\omega(\T,\fso(3,\R))$ with $||\ph||_h\leq\e$, $d\in\Z, c\in\R$. Then there exists $\psi, F\in C_{h'}^\omega(\T,\C)$, $h'<h$, such that
\begin{eqnarray*}&&\exp(\psi(x+\a))\exp(2\pi(dx+c)J_1)\exp(\ph(x))\exp(-\psi(x)) \\
&=&\exp(2\pi(dx+c)J_1+\cP_0\ph)\exp(\cP_{\a,d,c}\ph(x)+F(x)),
\end{eqnarray*}
where $F=O_2(\ph,\psi,\cP_{\a,d,c}\ph)$ and the following estimates hold:
$$||P_{\a,d,c}(\ph)||_h\leq\cst. \frac{||\ph||_h}{h},\quad
||\psi||_{h'}\leq\cst.\frac{||\ph||_h}{h(h-h')}.$$
\end{cor}

\subsection{Proof of Theorem~\ref{thm.4.1}}
Let $D=2\pi dJ_1$ with $d\in\Z$,  $\{w_1,w_2,w_3\}$ be the basis given by adjoint action of $d$, namely
$$\exp(-Dx)w_k\exp(Dx)=e^{2\pi il_kx}w_k,$$
for some $l_k\in\Z$, $\ph\in C_h^\omega(\T,\fso(3,\R))$. We assume that $l_3=0$. Then we decompose $\ph$ with respect to $\{w_1,w_2,w_3\}$, namely $\ph=\sum_{k=1}^3\ph_kw_k$. We define
$\Pr\ph=\sum_{k=1}^2T_{l_k}\ph_kw_k,$ where
$$ T_{l}\psi=\begin{cases}
\sum_{m=0}^{-l-1}e^{2\pi imx}\hat\psi(m),\quad\text{if }l<0;\\
\sum_{m=-l+1}^{0}e^{2\pi imx}\hat\psi(m),\quad\text{if }l>0.
\end{cases}$$
We cite \cite[Lemma 6.1]{Ka15} and reformulate it as following. It is proved through the estimates on length.
\begin{lem}
\label{lem.length}
 There exists a positive constant $\cst$ such that for $c\in\R$, $\ph\in C_h^\omega(\T,\fso(3,\R))$ small enough, if $(\a,\exp(2\pi(dx+c)J_1)\exp\ph)$ is of dynamical degree $2\pi dJ_1$, then it satisfies
\be
||\Pr\ph||_{L^2}\leq c||(I-\Pr)\partial\ph||_{L^2}.
\ee
\end{lem}

\begin{proof}[Proof of Theorem~\ref{thm.4.1}]
By our assumption, we can write $A(x)=\exp(2\pi(dx+c)J_1)\exp\ph(x)$ for some $\ph\in C_h^\omega(\T,\fso(3,\R))$ with $||\ph||_h\leq\e$. By Corollary~\ref{cor.conj}, there exists $c_1\in\R,\psi_1, F_1\in C_{h_1}^\omega(\T,\fso(3,\R))$, $h_1=\frac{h}{2}+\frac{h}{4}$, such that for $\ph_1=\cP_{\a,d,c}\ph,$
\begin{eqnarray*}&&\exp(\psi_1(x+\a))\exp(2\pi(dx+c)J_1)\exp(\ph(x))\exp(-\psi_1(x))\\
&=&\exp(2\pi(dx+c_1)J_1)\exp(\ph_1(x)+F_1(x)),
\end{eqnarray*}
with estimates
$$||\ph_1||_{h_1}\leq \cst.\frac{||\ph||_h}{h},\quad ||F_1||_{h_1}\leq\cst.\frac{||\ph||_h^2}{h^2(h-h_1)^2}.$$
Then Lemma~\ref{lem.length} gives that
\be
\label{eq.ph.F}
||\Pr(\ph_1+F_1)||_{L^2}\leq c||(I-\Pr)\partial(\ph_1+F_1)||_{L^2}.
\ee
As $\ph_1=\cP_{\a,d,c}\ph$, we have $\ph_1=\Pr\ph_1$. Then \eqref{eq.ph.F} gives that
$$||\ph_1+\Pr F_1||_{L^2}\leq c||(I-\Pr)\partial F_1||_{L^2}\leq c||F_1||_h\leq\cst.\frac{||\ph||_h^2}{h^2(h-h_1)^2}.$$
Hence we deduce that
$$||\ph_1+F_1||_h\leq ||\ph_1+\Pr F_1||_h+||(I-\Pr)F_1||_h\leq c||\ph_1+\Pr F_1||_{L^2}+\cst.||F_1||_h\leq\cst.\frac{||\ph||_h^2}{h^2(h-h_1)^2}.$$

\medskip

Let $\ph^1=\ph_1+F_1, \ph^0=\ph, h_0=h$. Inductively, we can find $h_j=\frac{h}{2}+\frac{h}{2^{j+1}}$, $c_j\in\R$, $\ph^j,\psi_j\in C_{h_n}^\omega(\T,\fso(3,\R))$, $j=1,\cdots,n$ such that
\begin{eqnarray*}&&\exp(\psi_j(x+\a))\exp(2\pi(dx+c_{j-1})J_1)\exp(\ph^{j-1}(x))\exp(-\psi_j(x))\\
&=&\exp(2\pi(dx+c_j)J_1)\exp(\ph^j(x)),
\end{eqnarray*}
with estimates
$$||\ph^j||_{h_j}\leq\cst.\frac{||\ph^{j-1}||_{h_{j-1}}^2}{h_{j-1}^2(h_{j-1}-h_j)^2},\quad j=1,\cdots,n.$$
Now we deal with $\ph^n$. By Corollary~\ref{cor.conj}, there exists $c_{n+1}\in\R$, $ \psi_{n+1}, F_{n+1}\in C_{h_{n+1}}^\omega(\T,\fso(3,\R))$, $h_{n+1}=\frac{h}{2}+\frac{h}{2^{n+1}}<h_n$, such that for $\ph_{n+1}=\cP_{\a,d,c_n}\ph^n,$
\begin{eqnarray}
\label{eq.Y.j}
 && \exp(\psi_{n+1}(x+\a))\exp(2\pi(dx+c_n)J_1)\exp(\ph^n(x))\exp(-\psi_{n+1}(x))\\
\nonumber &=&\exp(2\pi(dx+c_{n+1})J_1)\exp(\ph_{n+1}(x)+F_{n+1}(x)),
\end{eqnarray}
with estimates
$$||\ph_{n+1}||_{h_{n+1}}\leq \cst.\frac{||\ph_n||_{h_n}}{h_n},\quad ||F_{n+1}||_{h_{n+1}}\leq\cst.\frac{||\ph^n||_{h_n}^2}{h_n^2(h_n-h_{n+1})^2}.$$
Then Lemma~\ref{lem.length} gives that
\be
\label{eq.ph.F.n}
||\Pr(\ph_{n+1}+F_{n+1})||_{L^2}\leq C||(I-\Pr)\partial(\ph_{n+1}+F_{n+1})||_{L^2}.
\ee
As $\ph_{n+1}=\cP_{\a,d,c}\ph_n$, we have $\ph_{n+1}=\Pr\ph_{n+1}$. Then \eqref{eq.ph.F.n} gives that
$$||\ph_{n+1}+\Pr F_{n+1}||_{L^2}\leq C||(I-\Pr)\partial F_{n+1}||_{L^2}\leq C||F_{n+1}||_h\leq\cst.\frac{||\ph^n||_{h_n}^2}{h_n^2(h_n-h_{n+1})^2}.$$
Hence we deduce that
\begin{eqnarray*}||\ph_{n+1}+F_{n+1}||_h&\leq& ||\ph_{n+1}+\Pr F_{n+1}||_h+||(I-\Pr)F_{n+1}||_h \\
&\leq& C||\ph_{n+1}+\Pr F_{n+1}||_{L^2}+\cst.||F_{n+1}||_{h_{n+1}}\leq\cst.\frac{||\ph||_{h_n}^2}{h_n^2(h_n-h_{n+1})^2}.\end{eqnarray*}
Let $\ph^{n+1}=\ph_{n+1}+F_{n+1}$. Then we have
\be
\label{eq.e.n}
||\ph^{n+1}||_{h_{n+1}}\leq\cst.\frac{||\ph^n||_{h_n}^2}{h_n^2(h_n-h_{n+1})^2}=\cst.2^{2n}||\ph^n||^2_{h_n}.
\ee
Let $\e_n=||\ph^n||_{h_n}$. We have the following result for the convergence of the sequence.
\begin{lem}
Let the positive real sequence $\{\e_j\}$ satisfy
\be
\label{e.j}
\e_{j+1}\leq c4^j\e_j^2, j\in\N
\ee
 for some constant $c>0$. If $\e_0$ small enough, then there exists $\delta>0$ such that
$$\e_j\leq e^{-2^{j}\delta},\quad j\in\N.$$
\end{lem}
\begin{proof}
Let $v_j=\ln\e_j$, $\eta_j=\ln(c4^j)=c+j\ln4$. Then \eqref{e.j} gives
$$v_{j+1}\leq 2v_j+\eta_j.$$
Inductively we get
$$v_{j+1}\leq 2^{j+1}(\frac{\eta_j}{2^{j+1}}+\cdots+\frac{\eta_1}{2^2}+\frac{\eta_0}{2}+v_0).$$
Notice that $\sum_{j=0}^\infty\frac{\eta_j}{2^{j+1}}$ converges. Denote the limit by $\eta_\infty$. Choose $v_0<-\eta_\infty$, then
$$v_{j}<2^j(v_0+\eta_\infty)<0,\quad j\in\N.$$
Let $\delta=|v_0+\eta_\infty|>0$. Therefore, $\e_j=e^{v_j}\leq e^{-2^{j}\delta}$,which completes the proof.
\end{proof}

We choose $\e$ be small enough. Then we could inductively define $\ph^{n}\in\cB_{h_n}(\e)$ satisfying
$$||\ph^n||_{h_{n}}\leq e^{-2^{n}\delta},\quad n\in\Z.$$Therefore,
\be
\label{es.Y.j}
||\psi_{n+1}||_{h_{n+1}}\leq\cst.\frac{||\ph^n||_{h_n}}{h_n(h_n-h_{n+1})}
\leq\cst.2^ne^{-2^{n}\delta},\quad n\in\Z.
\ee

Let $\ti \psi_{n}\in C_{h_{n}}^\omega(\T,\fso(3,\R))  $ such that $\exp(\ti \psi_{n})=\exp(\psi_{1})\cdots\exp(\psi_{n})$. Then \eqref{es.Y.j} gives that $\ti\psi_{n}$ converges in $C_{\frac{h}{2}}^\omega(\T,\fso(3,\R))$. We denote the limit by $\ti \psi$. As $||c_{n+1}-c_n||=||\cP_0\ph^n||\leq||\ph^n||_{h_n}$, $c_n$ converges. We denote the limite by $\ti c$. By \eqref{eq.Y.j},
$$\exp(-\ti \psi(x+\a))\exp(2\pi(dx+c)J_1)\exp(\ph(x))\exp(-\ti \psi(x))=\exp(2\pi(dx+\ti c)J_1),$$
i.e., $(\a,A)$ is $C_{\frac{h}{2}}^\omega$-conjugate to $(\a,\exp(2\pi(dx+\ti c)J_1))$.
\end{proof}

\section{Proof of main results}
\label{sect.main}

\subsection{Proof of Theorem \ref{thm.ac.1}} In this subsection, we establish the equivalence of acceleration and degree.  We firstly study the accelerations along renormalization, which can be seen as generalization of Corollary 14 of \cite{A15}.
\begin{lem}
\label{w.n}
Let $\a\in\R\backslash\Q$, $G$ be a subgroup of $\GL(s,\R)$, $s\in\N^*$, $A\in C_h^\omega(\T,G)$, $(\a_n,\ti A^{(n)})$ be a representative of $n^{th}$ renormalization, $n\in\N$. Then
$$
L^i(\a_n,\ti A^{(n)}_\e)=\frac{1}{\beta_{n-1}}L^i(\a,A_{\beta_{n-1}\e}),\quad i=1,2,\cdots,s.
$$
Consequently, we have
$$w^i(\a,A)=w^i(\a_n,\ti A^{(n)})
,\quad i=1,2,\cdots,s.$$
\end{lem}
\begin{proof}
Let $N\in C^\omega_h(\R,\GL(s,\R))$ be a normalizing map of $n$th renormalization of $(\a,A)$, i.e.,
$$N(x+1)A_{(-1)^{n-1}q_{n-1}}(x_0+\beta_{n-1}x)N(x)^{-1}=I,$$
$$N(x+\a_n)A_{(-1)^nq_n}(x_0+\beta_{n-1}x)N(x)^{-1}=\ti A^{(n)}.$$
Then for  any $k,l\in\Z$.
\be
\label{eq.N}
A_{k(-1)^nq_n+l(-1)^{n-1}q_{n-1}}(x_0+\beta_{n-1}x)=N(x+k\a_n+l)^{-1}\ti A^{(n)}_kN(x),
\ee
 As there exists a canonical ring endomorphism from $\End(\R^s)$ to $\End(\Lambda^k\R^s)$,
we have for $P,Q\in\GL(s,\R)$,
$$\Lambda^i(PQ)=\Lambda^iP\cdot\Lambda^iQ.$$
Hence by \eqref{eq.N}, we have
$$(\Lambda^iA)_{k(-1)^nq_n+l(-1)^{n-1}q_{n-1}}(x_0+\beta_{n-1}x)=(\Lambda^iN(x+k\a_n+l))^{-1}(\Lambda^i\ti A^{(n)})_k\Lambda^iN(x).$$

Let $\e_0>0$ be such that $\ti A^{(n)}\in C^\omega_{\e_0}(\T,G)$ and $N$ admits an analytic extension to a open neighborhood of $\R$ containing $V=[0,2]\times[-\e_0,\e_0]$. Let $\Lambda_0=\sup_{z\in V}||\Lambda^iN(z)||^2.$ For $k\in\Z$, let $l=l(k)$ be the unique integer such that $0\leq k\a'+l<1$ and $t=t(k)=k(-1)^nq + l(-1)^{n-1}q_{n-1}$, then we have
$$\Lambda_0^{-1}\leq\frac{\|(\Lambda^iA)_t(y+\beta_{n-1}\e i)\|}{\|(\Lambda^i\ti A^{(n)})_k(x+\e i)\|}\leq \Lambda_0,$$
where $x,y\in\C/\Z$ are related by $y=x_0+\beta_{n-1}x$ and we assume that $|\Im(x+\e i)|<\e_0.$
It follows that
\begin{align*}
&|\int_{\T}\ln||(\Lambda^i\ti A^{(n)})_k(x+\e i)||dx-\int_{\T}\ln||(\Lambda^iA)_{(-1)^nt}(x+\beta_{n-1}\e i)||dx|\\
=&|\int_{\T}\ln||(\Lambda^i\ti A^{(n)})_k(x+\e i)||dx-\int_{\T}\ln||(\Lambda^iA)_{t}(x+\beta_{n-1}\e i)||dx|\leq\ln \Lambda_0.
\end{align*}
Notice that when $k$ is large, $t$ satisfies
$$(-1)^n\frac{t}{k}=q_n-\frac{l}{k}q_{n-1}=q_n+\a_nq_{n-1}+o(1)=\frac{1}{\beta_{n-1}}+o(1).$$
It follows that for large $k$,
$$\frac{1}{k}\int_{\T}\ln||(\Lambda^i\ti A^{(n)})_k(x+\e i)||dx=\frac{1+o(1)}{\beta_{n-1}}\frac{1}{(-1)^n}\int_{\T}\ln||(\Lambda^iA)_{(-1)^nt}(x+\beta_{n-1}\e i)||dx.$$
Taking the limit, we get
$$
L^i(\a_n,\ti A^{(n)}_\e)=\frac{1}{\beta_{n-1}}L^i(\a,A_{\beta_{n-1}\e}),\quad i=1,2,\cdots,s,
$$
which immediately implies $w^i(\a,A)=w^i(\a_n,\ti A^{(n)})$.
\end{proof}

With preceding lemma, now we can prove Theorem \ref{thm.ac.1} by renormalization.
\begin{proof}[Proof of Theorem \ref{thm.ac.1}]
Let $\ti A^{(n)}_t=\ti A^{(n)}(\cdot+it)$, $t\in\R$. We rewrite the convergence of renormalization from Theorem~\ref{sect4.0.2} to be
\be
\label{cv.t}
\ti A^{(n_k)}_t\xrightarrow{C^\omega}\exp(C+D(x+it))\quad \text{ on }\R,\quad t\in[-h,h],
\ee
for some subsequence of even integers $\{n_k\}$ and $C\in\fz(D)$. We assume that $L^i(\a,A)$ is affine on $(0,h_0)$ for some $h_0>0$. By Lemma \ref{w.n}, $L^i(\a_n,\ti A^{(n)})$ is also affine on $(0,h_0)$ and
$$\omega^i(\a_n,\ti A^{(n)})=\omega^i(\a,A)=:\omega^i,\quad i=1,\cdots,s,\quad n\in\N.$$
By compactness of $\SO(s,\R)$, $L_i(\a_n,\ti A^{(n)})=0$. We deduce that for $t\in(0,h_0)$,
$$L_i(\a_n,\ti A^{(n)}_t)=2\pi t\omega_i,\quad i=1,\cdots,s,\quad n\in\N.$$
By continuity of Lyapunov exponent \cite{AJS14}, \eqref{cv.t} gives that
\be
\label{w.d}
L_i(\a,\exp(C+D(x+it)))=\lim_{n\to\infty}L_i(\a_n,\ti A^{(n)}_t)=2\pi t\omega_i,\quad i=1,\cdots,s,\quad t\in(0,h_0).
\ee

\smallskip

Meanwhile, we recall the definition of Lyapunov exponents through singular value. As $C,D\in\fso(s,\R)$ commute with each other, $C$ will not contribute in $L_i(\a,\exp(C+D(x+it)))$. Moreover, compactness shows that there exist $Q\in\SU(s)$ and integers $d_1\geq\cdots\geq d_s$ such that $D=Q\cdot2\pi i\diag(d_1,\cdots,d_s)\cdot Q^{-1}$. We then compute explicitly that
$$L_i(\a,\exp(C+D(x+it)))=2\pi td_i,\quad i=1,\cdots,s.$$
Comparing with \eqref{w.d}, we deduce that
$$(\omega_1,\cdots,\omega_s)=(d_1,\cdots,d_s),$$
which completes the proof of Theorem \ref{thm.ac.1}.
\end{proof}

\subsection{Proof of Theorem \ref{thm.ac.2} (1) $\omega_1=0$}
\label{sect.5.2}
Now we start to prove Theorem \ref{thm.ac.2}. Since the case when $\omega_1$ zero or not has different dynamical nature, we separate the proof of Theorem \ref{thm.ac.2} into two subsections. We deal with $\omega_1=0$ here. To study almost reducibility along renormalization, we need the following lemma. 

\begin{lem}
\label{cor.a.r}
Let $h>0$, $\a\in\R\backslash\Q$,  $A\in C_h^\omega(\T,\SO(3,\R))$, $n\in\N$. If a representative of $n$th deep renormalization $(\a_n,\ti A^{(n)})$ satisfies $||\ti A^{(n)}-A_0||_h<\delta(h)$ for some constant $A_0\in\SO(3,\R)$, then $(\a,A)$ is $C^\omega$-almost reducible and accumulated by reducible cocycles.
\end{lem}
\begin{proof}
We apply Theorem ~\ref{thm.loc.dis} for $h>0$ and $L=4$, denoting $\delta(h,4)$ briefly by $\delta(h)$. when $||\ti A^{(n)}-A_0||_h<\delta(h)$, Theorem ~\ref{thm.loc.dis} gives that $(\a_n,\ti A^{(n)})$ is $C^\omega$-almost reducible. More precisely, for $j\in\N$, there exist $h_j>0$, $B_j\in C_{h_j}^\omega(\T,\SO(3,\R))$, $C_j\in\SO(3,\R)$,
 such that for $\ti A_j= B_j(\cdot+\a_n)\ti A^{(n)}(\cdot)B_j(\cdot)^{-1}$, $\e_j=\|B_j\|_{h_j}^4\cdot||\ti A_j-C_j||_{h_j}$, we have
\be
\label{bj}
\e_j\rightarrow 0,\qquad \hbox{as $j\rightarrow \infty$.}
\ee
We will work with large $j$ so that $\e_j$ small.

Let $(\a_n,\ti A_j)^m=(m\a_n,\ti A_{j,m}), m\in\N$. As
\begin{align*}
||\ti A_{j,m+1}(x)-C_j^{m+1}||_{h_j}&=||\ti A_j(x+m\a)(\ti A_{j,m}(x)-C_j^m)+(\ti A_j(x+m\a)-C_j)C_j^m||_{h_j}\\
&\leq(1+\e_j)||\ti A_{j,m}(x)-C_j^m||_{h_j}+\e_j,
\end{align*}
inductively we deduce that $||\ti A_{j,m}(x)-C_j^m||\leq(1+\e_j)^m-1$. We choose $j$ large enough so that $\e_j<<\frac{1}{q_n}$. Then we have
\be
\label{dist.q}
||\ti A_{j,q_n}-C_j^{q_n}||_h\leq q_n\e_j.
\ee

Now we define the distance for commuting pairs. We start with a notation of maximal norm. Let $I\subset\R$ be an interval, $h>0$, $f:I\times i[-h,h]\to\SO(3,\R)$ be a continuous map. We define
$$||f||_{h,I}=\sup_{x\in I\times i[-h,h]}||f(x)||.$$
Let $\Lambda^\omega_{h,I}$ be the set of all $\Phi\in\Lambda^\omega$ with $A^\Phi_{1,0}, A^\Phi_{0,1}$ analytically defined on a set containing $I\times i[-h,h]$. For $\Phi\in\Lambda^\omega_{h,I}$, we denote by $||\Phi||_{h,I}$ the quantities
$$||\Phi||_{h,I}:=\sup(||A^\Phi_{1,0}||_{h,I},||A^\Phi_{0,1}||_{h,I}),$$
 and then define a distance on $\Lambda^\omega_{h,I}$ as
$$d_{h,I}(\Phi_1,\Phi_2):=\sup(||A_{1,0}^{\Phi_1}-A_{1,0}^{\Phi_2}||_{h,I}, ||A_{0,1}^{\Phi_1}-A_{0,1}^{\Phi_2}||_{h,I}),\quad \Phi_1,\Phi_2\in\Lambda^\omega_{h,I}.$$
If $\Phi,\Phi'\in\Lambda_{h,I}^\omega$ are $C^\omega$-conjugated via $B\in C^{\omega}(\R,\SO(s,\R))$ analytically defined on a set containing $I\times i(-h,h)$, then
$$||\Phi'||_{h,I}\leq||B^{-1}||_{h,I}\max(||B||_{h,I+\gamma^\Phi_{1,0}},||B||_{I+\gamma^\Phi_{0,1}})||\Phi||_{h,I}.$$
Let $\Phi_1,\Phi_2\in\Lambda_{h,I}^\omega$, $B\in C^{\omega}(\R,\SO(3,\R))$ analytically defined on a set containing $I\times i(-h,h)$, $\Phi_i'=\Conj_B(\Phi_i)$, $i=1,2$, then
\be
\label{dist.conj}
d_{h,I}(\Phi_1',\Phi_2')\leq||B^{-1}||_{h,I}\max(||B||_{h,I+\gamma^\Phi_{1,0}},||B||_{h,I+\gamma^\Phi_{0,1}})d_{h,I}(\Phi_1,\Phi_2).
\ee

Let $\Phi_1=\Conj_{B_j}\bm1,I\\\a_n,\ti A^{(n)}\em$, $\Phi_2=\bm1,I\\\a_n,C_j\em$. Then $d_{h_j,I}^\omega(\Phi_1,\Phi_2)\leq\e_j$ for any interval $I\subset\R$. It is easy to see that
\be
d_{h_j,\beta_{n-1}I}^\omega(M_{\beta_{n-1}^{-1}}(\Phi_1),M_{\beta_{n-1}^{-1}}(\Phi_2))=\beta_{n-1}^{-1}d_{h_j,I}^\omega(\Phi_1,\Phi_2).
\ee
By \eqref{dist.q},
\be
\label{dist.1}
d_{h_j,\beta_{n-1}I}
^\omega(N_{Q_n^{-1}}\circ M_{\beta_{n-1}^{-1}}(\Phi_1),N_{Q_n^{-1}}\circ M_{\beta_{n-1}^{-1}}(\Phi_2))\leq q_n\beta_{n-1}^{-1}\e_j.
\ee

Assume that $N\in C^\omega(\R,\SO(3,\R))$ is the normalizing map such that
$$\bm1,I\\\a_n,\ti A^{(n)}\em=\Conj_N\circ M_{\beta_{n-1}}\circ N_{Q_n}\bm1,I\\\a,A\em.$$
Let $\Phi_3:=\Conj_{N^{-1}(\beta_{n-1}^{-1}\cdot)}\circ\Conj_{B_j^{-1}(\beta_{n-1}^{-1}\cdot)}\circ N_{Q_n^{-1}}\circ M_{\beta_{n-1}^{-1}}(\Phi_2)$, $I=\beta_{n-1}^{-1}[-\frac{1}{2},1]$. Then by \eqref{dist.conj} and \eqref{dist.1},
\begin{align}
\label{dist.2}
d_{h_j,[-\frac{1}{2},1]}
^\omega(\bm1,I\\\a,A\em,\Phi_3)
\leq ||N|_{[0,2\beta_{n-1}^{-1}]}||_{h_j}^2q_n\beta_{n-1}^{-1}||B_j||_{h_j}^2\e_j=:\bar\e_j,
\end{align}
By (\ref{bj}), 
$\bar\e_j\to0$ as $j\to\infty$. We work with $j$ large so that $\bar\e_j$ small enough.

\medskip
To continue analysis, we recall the quantitative normalizing lemma, based on normalizing lemma in \cite{AK06}, with a detailed proof in \cite{P23}.
\begin{lem}
\label{norm.quant}
For any $\Phi\in\Lambda_{0}^\omega$, $\Phi$ is $C^\omega$-conjugated to a normalized action. Moreover, there exist an $\e>0$ and a constant $\cst$ such that for any $\Phi\in\Lambda_{0}^\omega$ satisfying
$||A_{1,0}^\Phi-I||_{h,{[-\frac{1}{2},1]}}\leq\e$, for some $h>0$, we can construct a normalizing map $B\in C^\omega(\R,\SO(3,\R))$ such that 
$$B(\cdot+1)A_{1,0}^\Phi(\cdot)B(\cdot)^{-1}=I,\quad ||B-I||_{h,[0,2]}\leq\cst.||A_{1,0}^\Phi-I||_{h,[-\frac{1}{2},1]}.$$
\end{lem}

\medskip
Notice that $\Phi_3$ is not necessarily normal. Since \eqref{dist.2} gives that
$$||A^{\Phi_3}_{1,0}-I||_{h_j}\leq\bar\e_j,$$
we can construct from Lemma~\ref{norm.quant} a conjugacy $\ti B\in C^\omega_{h_j}(\R,\SO(3,\R))$ such that $$\ti B(\cdot+1)A^{\Phi_3}_{1,0}\ti B(\cdot)^{-1}=I$$ and satisfying
$$||\ti B-I||_{h_j,[0,2]}\leq \cst.\bar \e_j.$$
 As we can express $\ti B^{-1}(x)$ by entries of $\ti B(x)$, we get
$$||\ti B^{-1}-I||_{h_j,[0,2]}\leq \cst.\bar \e_j.$$
Then
\begin{align*}
&||\ti B(x+\a)A^{\Phi_3}_{0,1}(x)\ti B(x)^{-1}-A^{\Phi_3}_{0,1}||_{h_j,[0,1]}\\=&||\ti B(x+\a)A^{\Phi_3}_{0,1}(\ti B(x)^{-1}-I)+(\ti B(x+\a)-I)A_{0,1}^{\Phi_3}(x)||_{h_j,[0,1]}\leq \cst.\bar\e_j,
\end{align*}
as $\bar\e_j$ is small enough. Together with \eqref{dist.2}, we deduce that
\be
\label{dist.3}
d_{h_j,[0,1]}(\bm1,I\\\a,A\em,\Conj_{\ti B}\Phi_3)
\leq \cst.\bar\e_j.
\ee
This means that there exist $\breve{A}_j\in C^\omega_{h_j}(\T,\SO(s,\R)),  \breve{B}_j\in C^\omega_{h_j}(\T,\SO(s,\R))$, where $\breve{B}_j(\cdot)=\ti B_j(\cdot)B_j(\beta_{n-1}^{-1}\cdot)N(\beta_{n-1}^{-1}\cdot)$ satisfies
\be
\label{B.j}
||\breve{B}_j||_{h_j}\leq \cst.||B_j||_{h_j},
\ee
such that
\be
\label{A.A.j}
||A-\breve{A}_j||_{h_j}\leq \cst.\bar\e_j
\ee
and
\be
\label{bar.bar.A}
(\a,\breve{A}_j)=(0, \breve{B}_j)\circ(\a,C_j)\circ(0, \breve{B}_j)^{-1}.
\ee
Thus $A$ is approximated by reducible cocycles.
\medskip
By \eqref{B.j}, \eqref{A.A.j}, \eqref{bar.bar.A} and the convexity inequality, we have
$$||\breve{ B}_j(x+\a)^{-1}A(x)\breve{ B}_j(x)-C_j||_{h_j}\leq\cst.||B_j||_{h_j}^2\bar\e_j.$$
By  (\ref{bj}),  
$||B_j||_{h_j}^2\bar\e_j=||B_j||_{h_j}^4\e_j\to0$ as $j\to\infty$. Hence $(\a,A)$ is $C^\omega$-almost reducible.
\end{proof}

\begin{proof}[Proof of Theorem \ref{thm.ac.2} (1)] We prove the equivalence of $\omega_1=0$ and almost reducibility for analytic $\SO(3,\R)$-cocycles. 

If $(\a,A)$ is  $C^\omega$-almost-reducible, i.e., there exists $B_n\in C^\omega(\T,\SO(3,\R))$, $A_0\in\SO(3,\R)$ such that
$$B_n(x+\a)^{-1}A(x)B_n(x)\to A_0.$$
As acceleration is invariant by conjugation and upper semi-continuous in $C^\omega$-space, we deduce that
$\omega_1=0.$

Conversely, if $\omega_1=0$, by Theorem \ref{thm.ac.1}, the dynamical degree of $(\a,A)$  is zero. Then Theorem \ref{sect4.0.2} provides the convergence of renormalization to constant.
In other word, for $\delta(h)>0$, there exists $n_0\in\N$, such that $\ti A^{(n)}$ is $\delta(h)$ close to constant. 
Lemma~\ref{cor.a.r} therefore shows that $(\a,A)$ is $C^\omega$-almost reducible.

\end{proof}

\subsection{Proof of Theorem \ref{thm.ac.2} (2) $\omega_1\neq0$}

To prove the main theorem, we will find a conjugation representative when $\omega_1\neq0$. As conjugation is preserved along renormalization, we firstly compute how the normal form change along renormalization.

\begin{lem}
\label{ren.nor.form}
Let $\a\in\R\backslash\Q$, $D\in\fso(3,\R)$ satisfying $\exp D=I$, $C\in\fz(D)$. Then the cocycle $(\a,\exp(Dx+C))$ has representative of renormalization $(\{\frac{1}{\a}\},\exp(-Dx-D(\frac{1}{2\a}+\frac{1}{2})-\frac{C}{\a}))$.
\end{lem}
\begin{proof}
Let $a=a_1=\frac{1}{\a}-\{\frac{1}{\a}\}, \Upsilon(x)=\exp(Dx+C).$ We start with the commuting pair
 \be
\label{com.pa}
M_\a\circ\bm0&1\\1&-a\em\bm(1,I)\\(\a,H)\em=\bm(1, \Upsilon(\alpha x))\\(\frac{1}{\a},I)\circ(1,\Upsilon(\a x))^{-a}\em.
\ee
We define $J(x)=\exp(-\frac{\a}{2}D(x^2-x)-Cx)$. Then
\be
\label{eq.conj.J.D}
\Upsilon(\alpha x)=J(x+1)^{-1}J(x),
\quad
\Conj_{J} (1,\Upsilon(\alpha\cdot))=(1,I).
\ee
 Hence $J$ is a normalizing map of the one-step renormalization \eqref{com.pa}.

Now, inversing \eqref{eq.conj.J.D} and computing its $a$ power, we get
$$
\Conj_{J}(1,\Upsilon(\a\cdot))^{-a}=(1,I)^{-a}=(-a,I).
$$
Then we compute 
\begin{align}
\label{eq.ti.D}
\Conj_{J}((\frac{1}{\a},I)\circ(1,\Upsilon(\a \cdot))^{-a})
=&\Conj_{J}(\frac{1}{\a},I)\circ\Conj_{J}(1,\Upsilon(\a \cdot))^{-a}\nonumber\\
=&(0,J)\circ (\frac{1}{\a},I)\circ(0,J)^{-1}\circ(-a,I)\nonumber\\
=&(\{\frac{1}{\a}\},J(x+\{\frac{1}{\a}\})J(x-a)^{-1}).
\end{align}
Let
\begin{align}
\label{eq.ti.D.def}
\ti \Upsilon(x)=&J(x+\{\frac{1}{\a}\})J(x-a)^{-1}=\exp(-D(x+\frac{1}{2}+\frac{1}{2\a})-\frac{C}{\a}).
\end{align}
We then get a representative $(\{\frac{1}{\a}\},\ti \Upsilon)$ of one-step renormalization of $(\a,\Upsilon)$,
$$\bm(1,I)\\(\{\frac{1}{\a}\},\ti \Upsilon(x))\em=\Conj_{J(\a,d,c)}\circ M_\a\circ\bm0&-1\\-1&a\em\bm(1,I)\\(\a,\Upsilon)\em.$$
\end{proof}

\begin{proof}[Proof of Theorem \ref{thm.ac.2} (2)] 
 When $\omega_1\neq0$, by Theorem~\ref{thm.ac.1}, we only need to prove that if dynamical degree $D$ is nonzero, then the cocycle $(\a,A)$ is $C^\omega$-conjugate to $(\a,\exp(Dx+C))$ for some $C\in\fz(D)$.

By Theorem \ref{sect4.0.2}, we have convergence of renormalization, namely for a sequence of $C_n\in\fz(D)$, $\ph_n\in C_h^\omega(\T,\fso(3,\R))$,
$$\ti A^{(n)}(x)=\exp(\ph_n(x))\exp((-1)^nDx+C_n),$$
\be
\label{cvC2}
\lim_{n\to\infty}||\ph_n||_h=0.
\ee
As $\a\in \RDC$, there exist $\kappa>0,\tau>0$ such that $\a_n\in \DC(\kappa,\tau)$ for infinitely many $n\in\N$. For such $n$ large enough, by \eqref{cvC2}, Theorem~\ref{thm.4.1} gives that $(\a_n,\ti A^{(n)})$ is $C^\omega$-conjugate to $(\a_n,\exp((-1)^nDx+C'))$ for some $C'\in\fz(D)$. 

By Lemma~\ref{ren.nor.form}, the form $(\a,\exp(\pm Dx+\cdot))$ with $\cdot\in\fz(D)$ is preserved by renormalization. Hence there exists $C\in\fz(D)$ such that $(\a,\exp(Dx+C))$ has a representative of $n$th renormalization as $(\a_n,\exp((-1)^nDx+C'))$. 

In other word, $(\a,A)$ has a representative of $n$th renormalization $(\a_n,\ti A^{(n)})$ conjugate to $(\a_n,\exp((-1)^nDx+C'))$, a representative of $n$th renormalization of $(\a,\exp(Dx+C))$. Then Lemma~\ref{lem.ren.conj} gives that $(\a,A)$ is $C^\omega$-conjugate to $(\a,\exp(Dx+C))$. Hence we complete the proof of Theorem \ref{thm.ac.2}.
\end{proof}

\appendix
\section{Floquet Theory}
We introduce the following quantitative  Floquet Theory. In \cite{HY12}, we have introduced a quantitative  Floquet Theory for
quasi-periodic linear system on $\SL(2,\R)$. Here we consider the case of $\SO(3,\R)$, but the conclusion and the proof also
hold for general compact matrix group.
\begin{lem}
\label{F-Lem}
Let $\alpha\in\R\backslash \Q$,  suppose that
$F=\sum_{l\in\Z}\hat F(lq,-lp)e^{2\pi il(q\theta_1-p\theta_2)}\in \cB_h$.
Then there exists $B\in C^\omega_h(\T^2,\SO(3,\R))$ which
 conjugates the system
$$\begin{cases}
\dot{x}=F(\theta)x\\
\dot{\theta}=\omega=(\alpha,1)
\end{cases}$$ to the constant system $$\begin{cases}
\dot{x}=Cx\\
\dot{\theta}=\omega=(\alpha,1)
\end{cases}$$  with estimates
\be\label{est.-F-lem}
||B||_h\leq \exp \left\{\frac{(|q|+|p|) h}{|\tau|}||F||^\#_h\right\},
\qquad \tau:= q\alpha-p.
\ee
\end{lem}
\begin{proof}
Let $x=q\theta_1-p\theta_2, \ti h=(|q|+|p|)h$, $G(\phi)=\sum_{l\in\Z}\hat F(lq,-lp)e^{2\pi il \phi}$. Then we have $F(\theta_1,\theta_2)=G(q\theta_1-p\theta_2)$, and
\be\label{G-F-norm}
||G||_{\ti h}\leq\sum_{l\in\Z}||\hat F(lq,-lp)||e^{2\pi |l|\ti h}=||F||^\#_h.
\ee

Let $\Phi(t)$ be its basic matrix solution of the $\frac{1}{|\tau|}$-periodic ODE
\be
\label{ode}
\frac{dx}{dt}=G(\tau t)x
\ee
with $\Phi(0)=I$. Let $C:= |\tau|\log \Phi(\frac{1}{|\tau|})\in\fso(3,\R)$, i.e., $C$ satisfies
\be\label{C-form}
\Phi(\frac{1}{|\tau|})=\exp(\frac{1}{|\tau|}\cdot C),\qquad \|C\|\leq ||G||_{\ti h}\leq ||F||^\#_h.
\ee
We consider the complexification of (\ref{ode}), and consider the case of $t+is$ with some fixed $t\in \R$ and $|s|<\frac{\ti h}{|\tau|}$,
\be
\label{ode-img}
\frac{dx}{d(is)}=G(t+\tau s i)x.
\ee
Then, by Grownwall inequality and (\ref{G-F-norm}), as long as $|s|<\frac{\ti h}{|\tau|}$,
\be
\|\Phi(t+is)\|\leq \|\Phi(t)\|e^{|s|||G||_{\ti h}}\leq \|\Phi(t)\|e^{|s|||F||^\#_h}\leq e^{|s|||F||^\#_h},
\ee
($\|\Phi(t)\|=1$ for $t\in\R$ and $\Phi(t)$ is then $\SO(3,\R)-$valued). Then,
\be\label{Phi-est.}
\sup_{|\Im t|<\frac{\ti h}{|\tau|}}||\Phi(t)||\leq \exp\left\{\frac{||F||^\#_h \ti h}{|\tau|}\right\}
\leq \exp \left\{\frac{ (|q|+|p|) h}{|\tau|}||F||^\#_h\right\}.
\ee
And, for fixed $t\in \R$ and $|s|<\frac{\ti h}{|\tau|}$,  by (\ref{C-form})
\be
\|e^{(t+is)C}\|\leq \|e^{tC}\|\|e^{isC}\|\leq |\exp \{|s|\|C\|\}\leq \|e^{|s|||F||^\#_h}
\ee
($C\in \fso(3,\R)$ and $t\in \R$ means that $e^{tC}\in \SO(3,\R)$), which leas to
\be\label{exp-C-est.}
\sup_{|\Im t|<\frac{\ti h}{|\tau|}} \|e^{tC}\|\leq \exp \left\{\frac{||F||^\#_h \ti h}{|\tau|}\right\}\leq \exp \left\{\frac{(|q|+|p|) h}{|\tau|}||F||^\#_h\right\}.
\ee

Let $B_1(t)=\Phi(t)\exp(-Ct)$. By the definition of $C$, $B_1(t)=B_1(t+\frac{1}{|\tau|})$, and $B_1$ conjugate
\eqref{ode} to $\frac{dx}{dt}=C x$. It means
\be
\label{floquet-conj.}
B_1'(t)=B_1(t)G(\tau t)-CB_1(t).
\ee

Define $B(\theta_1,\theta_2):= B_1(\frac{1}{\tau}(q\theta_1-p\theta_2))\in C_h^\omega(\T^2,\fso(3,\R))$.
By (\ref{floquet-conj.}), it satisfies
\begin{eqnarray}
\partial_\omega B&=&\alpha \frac{\partial}{\partial \theta_1} B(\theta_1,\theta_2)+\frac{\partial}{\partial \theta_2} B(\theta_1,\theta_2)\nonumber\\
&=& \frac{q\alpha-p}{\tau}\frac{d}{dt}B_1'(\frac{1}{\tau}(q\theta_1-p\theta_2))=B_1'(\frac{1}{\tau}(q\theta_1-p\theta_2))\nonumber\\
&=&B_1(\frac{1}{\tau}(q\theta_1-p\theta_2))G(q\theta_1-p\theta_2)-CB_1(\frac{1}{\tau}(q\theta_1-p\theta_2))\nonumber\\
&=& B F -C B.
\end{eqnarray}
Thus  $B$ conjugate
$\begin{cases}
\dot{x}=F(\theta)x\\
\dot{\theta}=\omega
\end{cases}$ to $\,$ $\begin{cases}
\dot{x}=Cx\\
\dot{\theta}=\omega
\end{cases}$. By (\ref{Phi-est.}, \ref{exp-C-est.}),  (\ref{est.-F-lem}) follows
\end{proof}


\begin{thebibliography}{99}
\bibitem{Am09}
 Amor, Sana Hadj,
\textit{H\"{o}lder Continuity of the rotation number for quasi-periodic co-cycles in $\SL(2,\R)$},
 Communications in mathematical physics 2, no. 287 (2009): 565-588.


\bibitem{A3}
Avila, Artur,
\textit{Almost reducibility and absolute continuity I}, 
arXiv preprint (2010), arXiv:1006.0704.


\bibitem{A2}
Avila, Artur,
\textit{KAM, Lyapunov exponents, and the spectral dichotomy for one-frequency Schr\"odinger operators}, arXiv preprint (2023), arXiv:2307.11071.



\bibitem{A15}
Avila, Artur,
\textit{Global theory of one-frequency Schr\"odinger operators},
Acta Mathematica 215, no. 1 (2015): 1-54.

\bibitem{AJ10}
Avila, Artur, and Jitomirskaya, Svetlana,
\textit{Almost localization and almost reducibility},
 Journal of the European Mathematical Society 12, no. 1 (2009): 93-131.

\bibitem{AFK11}
Avila, Artur, Fayad, Bassam, and Krikorian, Rapha\"el,
\textit{A KAM scheme for $\SL(2,\R)$-cocycles with Liouvillean frequencies},
Geometric and Functional Analysis 21, no. 5 (2011): 1001-1019.

\bibitem{AJS14}
Avila, Artur, Jitomirskaya, Svetlana, and Sadel, Christian,
\textit{Complex one-frequency Cocycles},
Journal of the European Mathematical Society 16, no. 9 (2014): 1915-1935.

 
\bibitem{AJM}
Avila, Artur, Jitomirskaya, Svetlana, and Marx, Christoph A.,
\textit{Spectral theory of extended Harper’s model and a question by Erdős and Szekeres}, 
Inventiones mathematicae 210 (2017): 283-339.


   \bibitem{AKL}
Avila, Artur, Khanin, Konstantin, and Leguil, Martin,
\textit{Invariant graphs and spectral type of Schr\" odinger operators.}
arXiv preprint (2020), arXiv:2004.09137.


\bibitem{AK06}
Avila, Artur, and Krikorian, Rapha\"el,
\textit{Reducibility or nonuniform hyperbolicity for quasiperiodic Schr\"odinger cocycles},
Annals of Mathematics (2006): 911-940.


\bibitem{AK15}
Avila, Artur, and Krikorian, Rapha\"el,
\textit{Monotonic cocycles}, 
Inventiones mathematicae 202 (2015): 271-331.



\bibitem{AKP}
Avila, Artur, Krikorian, Rapha\"el, and Pan, Yi,
\textit{Renormalization of symplectic quasi-periodic cocycles},
In preparation.

\bibitem{AYZ}
Avila, Artur, You, Jiangong, and Zhou, Qi,
\textit{Sharp phase transitions for the almost Mathieu operator},
Duke Mathematical Journal 166 (2017): 2697-2718.


\bibitem{AYZ16}
Avila, Artur, You, Jiangong, and Zhou, Qi,
\textit{Dry ten Martini problem in the non-critical case},
 arXiv preprint (2023), arXiv:2306.16254.


\bibitem{Arnold} 
Arnol'd, Vladimir I.,
\textit{Small Denominators I: on the Mapping of a Circle into itself, AMS Transl.},  
In {\it Ser}, vol. 2, no. 46 (1965): 213-284. 


\bibitem{CCYZ19}
Cai, Ao, Chavaudret, Claire, You, Jiangong, and Zhou, Qi. 
\textit{Sharp H\"older continuity of the Lyapunov exponent of finitely differentiable quasi-periodic cocycles}, Mathematische Zeitschrift 291 (2019): 931-958.


\bibitem{De}
Denjoy, Arnaud,
\textit{Sur les courbes d\'efinies par les \'equations diff\'erentielles \`a la surface du tore},
 Journal de math\'ematiques pures et appliqu\'ees 11 (1932): 333-376.



\bibitem{DS75}
Dinaburg, Efim I., and Sinai, Yakov G.,
\textit{The one-dimensional Schrödinger equation with a quasi-periodic potential},
 Functional Analysis and Its Applications 9, no. 4 (1975): 279-289.


\bibitem{E92}
Eliasson, H\r{a}kan,
\textit{Floquet solutions for the one-dimensional quasiperiodic Schr\"{o}dinger equation},
 Communications in Mathematical Physics 146 (1992): 447-482.



\bibitem{E02}
Eliasson, H\r{a}kan,
\textit{Ergodic skew-systems on $\T^d\times\SO(3,\R)$},
Ergodic theory and dynamical systems 22, no. 5 (2002): 1429-1449.

\bibitem{E}
Eliasson, H\r{a}kan, 
\textit{Reducibility and point spectrum for linear quasi-periodic skew-products}, 
Proceedings of the International Congress of Mathematicians. Vol. II (1998): 779-787.


\bibitem{FK09}
Fayad, Bassam, and Krikorian, Rapha\"el,
\textit{Rigidity results for quasiperiodic $\SL (2, \R)$-cocycles.}
 J. Mod. Dyn 3, no. 4 (2009): 497-510.

\bibitem{FK}
Fayad, Bassam, and Krikorian, Rapha\"el, 
\textit{Some questions around quasi-periodic dynamics}, 
In Proceedings of the International Congress of Mathematicians (ICM 2018) (In 4 Volumes) Proceedings of the International Congress of Mathematicians (2018): 1909-1932. 


\bibitem{Fra00}
Fr\k{a}czek, Krzysztof,
\textit{On cocycles with values in the group $\SU(2)$},
Monaschefte f\"{u}r Mathematik(2000), 279-307.


\bibitem{Fra04}
Fr\k{a}czek, Krzysztof,
\textit{On the degree of cocycles with values in the group $\SU(2)$},
Israel Journal of Mathematics 139 (2004), 293–317.



\bibitem{GLL}
Gabriel, Patrick, Lemanczyk, Mariusz, and Liardet, Pierre,
\textit{Ensemble d'invariants pour les produits crois\'es de Anzai},
M\'emoires de la Socié\'eté\'e Math\'ematique de France 47 (1991): 1-102.


\bibitem{GY}
Ge, Lingrui, and You, Jiangong,
\textit{Arithmetic version of Anderson localization via reducibility}, 
Geometric and Functional Analysis 30 (2020): 1370-1401.


\bibitem{He79}
Herman, Michael R.,
\textit{Sur la conjugaison diff\'erentiable des diff\'eomorphismes du cercle \`a des rotations}, 
Publications Math\'ematiques de l'IH\'ES 49 (1979): 5-233.


\bibitem{He85}
Herman, Michael R.,
\textit{Simple proofs of local conjugacy theorems for diffeomorphisms of the circle with almost every rotation number}, 
Boletim da Sociedade Brasileira de Matem\'atica-Bulletin/Brazilian Mathematical Society 16, no. 1 (1985): 45-83.


\bibitem{H11}
Hou, Xuanji,
\textit{Ergodicity, minimality and reducibility of cocycles on some compact groups},
Taiwanese Journal of Mathematics 15, no. 3 (2011): 1247-1259.


\bibitem{HY12}
Hou, Xuanji, and You, Jiangong,
\textit{Almost reducibility and non-perturbative reducibility of quasi-periodic linear systems},
 Inventiones mathematicae 190, no. 1 (2012): 209-260.


\bibitem{HSY19}
Hou, Xuanji, Shan, Yuan, and You, Jiangong,
\textit{Construction of QuasiPeriodic Schr\"{o}dinger Operators with Cantor Spectrum},
Annales Henri Poincaré, vol. 20, pp. 3563-3601. Springer International Publishing, 2019.

\bibitem{ILR}
Iwanik, Anzelm, Lema\'nczyk, Mariusz, and Rudolph, Daniel,
\textit{Absolutely continuous cocycles over irrational rotations}, 
Israel Journal of Mathematics 83 (1993): 73-95.

\bibitem{Ka15}
Karaliolios, Nikolaos,
\textit{Global aspects of the reducibility of quasiperiodic cocycles in semisimple compact Lie groups},
M\'emoires de la Soci\'e\'e Math\'ematique de France 146, (2015).


\bibitem{Ka17}
Karaliolios, Nikolaos,
\textit{Differentiable rigidity for quasi-periodic cocyles in compact Lie groups},
Journal of Modern Dynamics 11 (2017): 125-142.


\bibitem{Ka18}
Karaliolios, Nikolaos,
\textit{Continuous spectrum or measurable reducibility for quasi-periodic Cocycles in $\T^d\times \SU(2)$},
Communications in Mathematical Physics 358 (2018): 741-766.


\bibitem{K99a}
Krikorian, Rapha\"el,
\textit{R\'educitibilit\'e presque partout des flots fibr\'es quasi-p\'eriodiques \`a valeurs dans les groupes compacts,}
Annales scientifiques de l'Ecole normale sup\'erieure, vol. 32, no. 2 (1999): 187-240.


\bibitem{K99b}
Krikorian, Rapha\"el,
\textit{R\'educitibilit\'e des syst\`emes produits-crois\'es \`a valuers dans des groupes compacts,}
Ast\'erisque, 259 (1999).



\bibitem{K01}
Krikorian, Rapha\"el,
\textit{Global density of reducible quasi-periodic cocycles on $\T^1 \times \SU(2)$,}
 Annals of Mathematics (2001): 269-326.

\bibitem{Kri02}
Krikorian, Rapha\"el,
\textit{Reducibility, differentiable rigidity and Lyapunov exponents for quasiperiodic cocycles on $\T\times\SL(2,\R)$},
arXiv preprint (2002), arXiv:math/0402333.


\bibitem{KO89a}
Katznelson, Yitzhak, and Ornstein, Donald,
\textit{The differentiability of the conjugation of certain diffeomorphisms of the circle}, 
Ergodic Theory and Dynamical Systems 9, no. 4 (1989): 643-680.


\bibitem{KO89b}
Katznelson, Yitzhak, and Ornstein, Donald,
\textit{The absolute continuity of the conjugation of certain diffeomorphisms of the circle},
 Ergodic Theory and Dynamical Systems 9, no. 4 (1989): 681-690.


\bibitem{KS87}
Khanin, Konstantin M., and Sinai, Yakov G.,
\textit{A new proof of M. Herman's theorem},
 Communications in Mathematical Physics 112 (1987): 89-101.





\bibitem{Ko84}
Kotani, Shinichi,
\textit{Ljapunov indices determine absolutely continuous spectra of stationary random one-dimensional Schr\"odinger operators},
North-Holland Mathematical Library, Elsevier, vol. 32 (1984): 225-247. 

\bibitem{MJ}
Marx, Christoph A., and Jitomirskaya, Svetlana,
\textit{Dynamics and spectral theory of quasi-periodic Schr\"odinger-type operators},
Ergodic Theory and Dynamical Systems 37, no. 8 (2017): 2353-2393.


\bibitem{MP84}
Moser, J\"urgen, and P\"oschel, J\"urgen,
\textit{An extension of a result by Dinaburg and Sinai on quasi-periodic potentials}, 
Commentarii Mathematici Helvetici 59 (1984): 39-85.


\bibitem{P23}
Pan, Yi,
\textit{Renormalization of symplectic quasi-periodic cocycles,}
PhD thesis (2023), www.theses.fr/2023CYUN1204.


\bibitem{Si83}
Simon, Barry,
\textit{Kotani theory for one dimensional stochastic Jacobi matrices}, 
Communications in mathematical physics 89 (1983): 227-234.


\bibitem{SK89}
Sinai, Yakov G., and Khanin, Konstantin M.,
\textit{Smoothness of conjugacies of diffeomorphisms of the circle with rotations}, 
Russian Mathematical Surveys 44, no. 1 (1989): 69-99.


\bibitem{WXYZ}
Wang, Jing.   Xu, Xu.  You, Jiangong and Zhou, Qi,
\textit{ Absolute continuity of the integrated density of states in the localized regime},
arXiv preprint (2023), arXiv:2305.00457.

\bibitem{Y84}
Yoccoz, Jean-Christophe,
\textit{Conjugaison differentielle des diff\'eomorphismes du cercle dont le nombre de rotation v\'erifie une
condition Diophantienne},
Annales scientifiques de l'\'Ecole Normale Sup\'erieure, vol. 17, no. 3 (1984): 333-359. 


\bibitem{Y}
Yoccoz, Jean-Christophe,
\textit{Analytic linearization of circle diffeomorphisms. Dynamical systems and small divisors},
in Dynamical Systems and Small Divisors: Lectures given at the CIME Summer School held in Cetraro, Italy, June 13-20, 1998 (2002): 125-173.




\bibitem{You}
You, Jiangong,
\textit{Quantitative almost reducibility and its applications},
In Proceedings of the International Congress of Mathematicians (ICM 2018) (In 4 Volumes) Proceedings of the International Congress of Mathematicians (2018): 2113-2135. 

\bibitem{YZ13}
You, Jiangong, and Zhou, Qi,
\textit{Embedding of analytic quasi-periodic cocycles into analytic quasi-periodic linear systems and its applications}, 
Communications in Mathematical Physics 323 (2013): 975-1005.


\end{thebibliography}
\end{document}